\documentclass[11pt]{amsart}
\usepackage{amsmath,amssymb,yhmath, url,color,booktabs}
\usepackage{fullpage}
\usepackage{tikz}
\usepackage{mathrsfs}
\usepackage{float}
\usepackage{microtype}
\usepackage{enumitem}
\usepackage{longtable}

\newtheorem{lemma}{Lemma}[section]
\newtheorem{prop}[lemma]{Proposition}
\newtheorem{cor}[lemma]{Corollary}
\newtheorem{thm}[lemma]{Theorem}

\newtheorem{thm?}[lemma]{Theorem?}

\newtheorem{conj}[lemma]{Conjecture}

\newtheorem{remark}[lemma]{Remark}
\newtheorem{questions}[lemma]{Questions}
\DeclareMathAlphabet{\curly}{U}{rsfs}{m}{n}
\def\Gg{\mathcal{G}}

\newcommand{\Z}{\mathbb{Z}}
\newcommand{\M}{\mathscr{M}}
\newcommand{\D}{\mathcal{D}}
\newcommand{\sD}{\mathcal{D}}
\newcommand{\sA}{\mathcal{A}}
\newcommand{\Hh}{\mathcal{H}}
\renewcommand{\d}{\mathbf{d}}
\newcommand{\N}{\mathscr{N}}

\newcommand{\Q}{\mathbb{Q}}

\newcommand{\End}{\operatorname{End}}

\newcommand{\Gal}{\operatorname{Gal}}

\newcommand{\OO}{\mathcal{O}}

\newcommand{\tors}{\operatorname{tors}}

\newcommand{\CM}{\operatorname{CM}}

\newcommand{\Ker}{\ker}
\newcommand{\E}{\mathscr{E}}

\newcommand{\Ok}{\mathcal{O}_K}

\newcommand{\Pp}{\mathscr{P}}
\newcommand{\T}{\mathscr{G}}
\newcommand\leg{\genfrac(){.4pt}{}}
\newcommand\rad{\operatorname{rad}}
\newlist{pcases}{enumerate}{1}
\setlist[pcases]{
  label=\textsc{Case~\arabic*:}\protect\thiscase.~,
  ref=\arabic*,
  align=left,
  labelsep=0pt,
  leftmargin=0pt,
  labelwidth=0pt,
  parsep=6pt
}
\newcommand{\case}[1][]{%
  \if\relax\detokenize{#1}\relax
    \def\thiscase{}%
  \else
    \def\thiscase{~#1}%
  \fi
  \item
}

\allowdisplaybreaks

\title{Torsion subgroups of CM elliptic curves over odd degree number fields}

\author{Abbey Bourdon}
\email{abourdon@uga.edu}

\author{Paul Pollack}
\email{pollack@uga.edu}

\address{Department of Mathematics\\Boyd Graduate Studies Building\\University of Georgia\\Athens, Georgia 30602\\USA}

\begin{document}
\begin{abstract} Let $\T_{\CM}(d)$ denote the collection of groups (up to isomorphism) that appear as the torsion subgroup of a CM elliptic curve over a degree $d$ number field. We completely determine $\T_{\CM}(d)$ for odd integers $d$ and deduce a number of statistical theorems about the behavior of torsion subgroups of CM elliptic curves. Here are three examples: (1) For each odd $d$, the set of natural numbers $d'$ with $\T_{\CM}(d')=\T_{\CM}(d)$ possesses a well-defined, positive asymptotic density. (2) Let $T_{\CM}(d) = \max_{G \in \T_{\CM}(d)} \#G$; under the Generalized Riemann Hypothesis,
  \[ \left(\frac{12e^{\gamma}}{\pi}\right)^{2/3} \le \limsup_{\substack{d\to\infty\\d\text{ odd}}} \frac{T_{\CM}(d)}{(d\log\log{d})^{2/3}} \le \left(\frac{24e^{\gamma}}{\pi}\right)^{2/3}. \]
(3) For each $\epsilon > 0$, we have $\#\T_{\CM}(d) \ll_{\epsilon} d^{\epsilon}$ for all odd $d$; on the other hand, for each $A> 0$, we have $\#\T_{\CM}(d) > (\log{d})^A$ for infinitely many odd $d$.
\end{abstract}

\date{\today}
\maketitle
\section{Introduction}
For a given positive integer $d$, let $\T(d)$ denote the set of (isomorphism classes) of abelian groups that appear as $E(F)[{\rm tors}]$ for some elliptic curve $E$ defined over some degree $d$ number field $F$, and let $T(d)$ denote the supremum of the orders of all such groups. Celebrated work of Merel \cite{merel} shows that $T(d)< \infty$ for every $d$. However, the nature of the finite sets $\T(d)$ remains largely mysterious. The only $d$ for which $\T(d)$ has been completely determined are $d=1$ (Mazur \cite{mazur77}, 1977) and $d=2$ (work of Kamienny, Kenku, and Momose, completed in 1992 \cite{KM88,kamienny92}). And while there are completely explicit upper bounds on $T(d)$, the known bounds grow superexponentially, whereas it is widely believed that $T(d)$ is bounded polynomially in $d$.

More can be said if we restrict the class of elliptic curves under consideration. In particular, elliptic curves with complex multiplication (CM) are of interest since they are known to provide examples of rational points of large order appearing in unusually low degree \cite{TOR1}. Let $\T_{\CM}(d)$ and $T_{\CM}(d)$ be defined as above, but with the added restriction that $E$ has CM.
Whereas $\T(d)$ is known only for $d=1$ and $d=2$, the set $\T_{\CM}(d)$ has been computed for all $d \le 13$ \cite{tor2}. And in contrast to the situation for $T(d)$ where the known upper bounds are (presumably) far from sharp, the upper order of $T_{\CM}(d)$ has recently been determined. In \cite{CP15}, it is shown that 
\[ \limsup_{d\to\infty} \frac{T_{\CM}(d)}{d\log \log{d}} < \infty. \]
From earlier work of Breuer \cite{breuer10}, this $\limsup$ is positive. Hence, $T_{\CM}(d)$ has upper order $d\log\log{d}$. 
Several other statistics concerning $T_{\CM}(d)$ are investigated in \cite{BCP}; e.g., it is shown there that the average of $T_{\CM}(d)$ for $d\le x$ is  $x/(\log{x})^{1+o(1)}$, as $x\to\infty$.

In \cite{BCS}, the authors study torsion of CM elliptic curves over \emph{real} number fields, meaning number fields admitting at least one real embedding. Observe that all number fields of odd degree are real. One of the central results of \cite{BCS} is a complete classification of which groups arise as torsion subgroups of CM elliptic curves defined over number fields of odd degree, i.e., a classification of the elements of $\bigcup_{d\text{ odd}}\T_{\CM}(d)$. This strengthens earlier work of Aoki \cite{aoki95}.
\begin{thm}[Odd Degree Theorem, \cite{BCS}, cf. {\cite[Corollary 9.4]{aoki95}}]
\label{ODDDEGREETHMINTRO}
Let $F$ be a number field of odd degree, let $E_{/F}$ be a $K$-CM elliptic curve, and let $T = E(F)[\tors]$.  Then: 
\begin{enumerate}
\item[a)] One of the following occurs:
\begin{enumerate}
\item[(1)] $T$ is isomorphic to the trivial group $\{ \bullet\}, \, \Z/2\Z, \, \Z/4\Z,$ or $\Z/2\Z \oplus \Z/2\Z$;
\item[(2)] $T \cong \Z/\ell^n \Z$ for a prime $\ell \equiv 3 \pmod{8}$ and $n \in \mathbb{Z}^+$ and
$K = \Q(\sqrt{-\ell})$;
\item[(3)] $T \cong \Z/2\ell^n \Z$ for a prime $\ell \equiv 3 \pmod{4}$ and $n \in \mathbb{Z}^+$ and
$K = \Q(\sqrt{-\ell})$.
\end{enumerate}
\item[b)] If $E(F)[\tors] \cong \Z/2\Z \oplus \Z/2\Z$, then $\End E$ has discriminant $\Delta = -4$. 
\item[c)] If $E(F)[\tors] \cong \Z/4\Z$, then $\End E$ has discriminant $\Delta \in \{ -4,-16\}$.
\item[d)] Each of the groups listed in part a) arises up to isomorphism as the torsion subgroup $E(F)$ of a CM elliptic curve $E$ defined over an odd degree number field $F$.
\end{enumerate}
\end{thm}
However, given a particular subgroup that does arise, the argument of \cite{BCS} does not identify the degrees $d$ in which it occurs. The main theorem of this paper is precisely such a result. Here, $h_{\Q(\sqrt{-\ell})}$ denotes the class number of $\Q(\sqrt{-\ell})$.

 \begin{thm}[Strong Odd Degree Theorem]\label{thm:strongodddegree}
 Let $\ell \equiv 3 \pmod{4}$ and $n \in \Z^+$. Define $\delta$ as follows:
 \[\delta = \begin{cases} 
      \left \lfloor{\frac{3n}{2}}\right \rfloor-1, & \ell > 3,\\
      0, & \ell=3 \text{ and } n=1,\\
      \left \lfloor{\frac{3n}{2}}\right \rfloor-2, & \ell=3 \text{ and } n \geq 2.
   \end{cases}\]\\
Then:
 \begin{enumerate}
 \item For any odd positive integer $d$, the groups $\{ \bullet\}, \, \Z/2\Z, \, \Z/4\Z,$ and $\Z/2\Z \oplus \Z/2\Z$ appear as the torsion subgroup of a CM elliptic curve defined over a number field of degree $d$.
\item $\Z/\ell^n\Z$ appears as the torsion subgroup of a CM elliptic curve defined over a number field of odd degree $d$ if and only if $\ell \equiv 3 \pmod{8}$ and $d$ is a multiple of $h_{\Q(\sqrt{-\ell})} \cdot \frac{\ell-1}{2}\cdot \ell^{\delta}$.
 \item $\Z/2\ell^n\Z$ appears as the torsion subgroup of a CM elliptic curve defined over a number field of odd degree $d$ if and only if one of the following holds:
 \begin{enumerate}
 \item[a.] $\ell \equiv 3 \pmod{8}$, where $n \geq 2$ if $\ell=3$, and $d$ is a multiple of $3 \cdot h_{\Q(\sqrt{-\ell})} \cdot \frac{\ell-1}{2}\cdot \ell^{\delta}$, or
 \item[b.] $\ell=3$ and $n=1$ and $d$ is any odd positive integer, or
 \item[c.] $\ell \equiv 7 \pmod{8}$ and $d$ is a multiple of $h_{\Q(\sqrt{-\ell})} \cdot \frac{\ell-1}{2}\cdot \ell^{\delta}$.
 \end{enumerate}
\end{enumerate}
\end{thm}
\noindent This theorem can be used to algorithmically determine $\T_{\CM}(d)$ for any odd degree $d$. See section 7 for a table of the groups which arise for odd $d \leq 99$.

The CM elliptic curves with a point of order $\ell^n$ in lowest possible odd degree are unexpectedly varied. In the case of even degrees, points of order $\ell$ often appear for the first time on a CM elliptic curve with a rational $j$-invariant. For example, any prime $\ell \equiv 1 \pmod{3}$ appears for the first time in even degree $\frac{\ell-1}{3}$ on an elliptic curve $E$ with $j(E)=0$ by \cite[Theorem 1]{TOR1}. However, we see already from the Odd Degree Theorem that a CM elliptic curve with a rational point of order $\ell$ in odd degree will not have a rational $j$-invariant once $\ell>163$. Moreover, through the proof of the Strong Odd Degree Theorem, we find that the algebraic structure of these optimal examples is surprisingly complex. Specifically, we find that if $n \geq 3$, an elliptic curve with a point of order $\ell^n$ in lowest possible odd degree necessarily has CM by a \emph{non-maximal} order, and the size of the conductor increases with $n$. See Remark \ref{order} for a precise statement along these lines. Non-maximal orders provide considerable technical complications, due in part to the fact that ideals do not necessarily factor uniquely into prime ideals, and so many results in the literature are formulated only to address the case of CM by the full ring of integers.  However, we know now that non-maximal orders play a crucial role in the extremal behavior of rational torsion points of elliptic curves.



Theorem \ref{thm:strongodddegree} opens the door to establishing new statistical properties of $\T_{\CM}(d)$ and $T_{\CM}(d)$, as $d$ ranges over odd integers. The mean of $T_{\CM}(d)$ for odd $d\le x$ is studied already in \cite{BCP}, where it is shown to be $x^{1/3+o(1)}$, as $x\to\infty$. This should be compared with the unrestricted average, which we recalled above was $x/(\log{x})^{1+o(1)}$. In the next theorem, we study the upper order of $T_{\CM}(d)$ for odd $d$. Again, we find that it is much smaller than the corresponding unrestricted statistic.

\begin{thm}[Upper order of $T_{\CM}(d)$ for odd $d$]\label{thm:upperodd}\mbox{ }
\begin{enumerate}
\item There is an infinite sequence of odd $d$, with $d\to\infty$, where
\[ T_{\CM}(d) \ge \left(\left(\frac{12e^{\gamma}}{\pi}+o(1)\right)d\log\log{d}\right)^{2/3}. \]
\item Assume the Riemann Hypothesis for Dirichlet $L$-functions. Then as $d\to\infty$ through all odd integers,
\[ T_{\CM}(d) \le \left(\left(\frac{24e^{\gamma}}{\pi}+o(1)\right)d\log\log{d}\right)^{2/3}. \]
Unconditionally, $T_{\CM}(d) \ll_{\epsilon} d^{2/3+\epsilon}$ for each fixed $\epsilon > 0$ and all odd $d$.
\end{enumerate}
\end{thm}

We turn next to understanding the ``stratification'' of torsion by degree.

In 1974, Olson showed that $\T_{\CM}(1) = \{\{\bullet\}, \Z/2\Z, \Z/3\Z, \Z/4\Z, \Z/6\Z, \Z/2\Z \oplus \Z/2\Z\}$ \cite{Olson74}. Computations
 carried out in \cite{tor2} revealed that this same list reoccurs as $\T_{\CM}(d)$ for several other small values of $d$. This was one of the phenomena investigated in \cite{BCS}, where it was shown that $\T_{\CM}(d)=\T_{\CM}(1)$ if $d=p$ or $p^2$ for some prime $p \ge 7$.  

Call $d$ an \emph{Olson degree} if $\T_{\CM}(d) = \T_{\CM}(1)$. The following complete classification of Olson degrees was proved in \cite{BCP}. To state the result, we need one more piece of notation. Given a set $\Gg$ of positive integers, we write $\M(\Gg)$ for the \emph{set of multiples of $\Gg$}, meaning the collection of all positive integers divisible by some element of $\Gg$. 

\begin{prop}[\cite{BCP}]\label{prop:BCP} The complement of the set of Olson degrees can be written as $\M(\Gg)$, where
\[ \Gg = \{2\} \cup \left\{\frac{\ell-1}{2} \cdot  h_{\Q(\sqrt{-\ell})}: \ell \text{ prime},\, \ell \equiv 3\pmod{4},\, \ell > 3\right\}. \]
\end{prop}

\noindent As a corollary of Proposition \ref{prop:BCP}, it was proved in \cite{BCS} that the set of Olson degrees possesses a positive density. 
It is easy, for reasons recalled at the start of \S\ref{sec:upper}, to rigorously bound this density from above by $4/15 = 0.26666\dots$. However, no explicit lower bound is proved in \cite{BCP}. We take the opportunity here to address this lacuna, estimating the density of Olson degrees to within $0.1\%$.

\begin{thm}\label{thm:olsondensity} The density of Olson degrees lies in the open interval $(0.264,0.265)$.
\end{thm}

\noindent Thus, a little more than half of the odd numbers $d$ are Olson degrees.

It is natural to wonder if our results on Olson degrees are the tip of a larger iceberg. Generalizing the above, we say that $d$ and $d'$ are \emph{CM-torsion-equivalent} if $\T_{\CM}(d) = \T_{\CM}(d')$. In this case, we call $d$ a $d'$-Olson degree.  The following questions were suggested by Pete L. Clark:

\begin{questions}\label{ques:PLC} Is it true that for every $d$, the set of $d$-Olson degrees possesses an asymptotic density? If so, is the sum of the densities of $d$-Olson degrees, taken over inequivalent $d$, equal to $1$?
\end{questions}
\noindent To avoid a possible source of confusion, we remind the reader that asymptotic density is finitely additive but not countably additive. Thus, an affirmative answer to the first question does not immediately imply an affirmative answer to the second.

Note that if $d$ and $d'$ are CM-torsion equivalent, then $d$ and $d'$ share the same parity. Indeed, if $d'$ is even, then $\Z/3\Z \oplus \Z/3\Z$ is realizable in degree $d'$ (see \cite[Theorems 1.4, 2.1]{BCS}), whereas the Odd Degree Theorem guarantees that such a group is never realizable in any odd degree $d$, so that $\T_{\CM}(d) \ne \T_{\CM}(d')$. Consequently, each equivalence class consists entirely of even integers or entirely of odd integers. 

Using Theorem \ref{thm:strongodddegree}, we are able to answer affirmatively the odd degree variants of Questions \ref{ques:PLC}.

\begin{thm}[Stratification of torsion in odd degrees]\label{thm:stratification} For each odd positive integer $d$, the set of $d$-Olson degrees possesses a positive asymptotic density. Moreover, using $\d(\cdot)$ for asymptotic density,
\[ \sum_{d} \d(\{d\text{-Olson degrees}\}) = \frac{1}{2}, \]
where the sum on the left is taken over any complete set of inequivalent odd integers.
\end{thm}

\begin{remark} One can use the method of proof of Theorem \ref{thm:olsondensity} to study the density of $d$-Olson degrees for other odd $d$. For instance, we have calculated in this way that the $3$-Olson degrees have density between $6.2\%$ and $6.4\%$. \end{remark}

We conclude with a result about the number of groups realizable in a given odd degree, i.e., the number $\#\T_{\CM}(d)$ for odd $d$. There is no mystery about how small $\#\T_{\CM}(d)$ can be; the Strong Odd Degree Theorem implies that $\T_{\CM}(1)$ is always a subset of $\T_{\CM}(d)$, so that $\#\T_{\CM}(d) \ge 6$, with equality if and only $d$ is an Olson degree. But how large can $\#\T_{\CM}(d)$ be? This is the subject of our final theorem, proved in \S\ref{sec:groupcounts}.
\begin{thm}\label{thm:maxgroupcount} For each fixed $\epsilon > 0$, we have $\#\T_{\CM}(d) \ll_{\epsilon} d^{\epsilon}$ for all positive odd integers $d$. On the other hand, for each fixed $A > 0$, there are infinitely many odd $d$ with $\#\T_{\CM}(d) > (\log{d})^A$.
\end{thm}
\noindent Theorem \ref{thm:maxgroupcount} provides another point of contrast between the even and odd degree cases. Indeed, at the end of \S\ref{sec:groupcounts} we will  adapt methods of Erd\H{o}s \cite{erdos35} and Pomerance \cite{pomerance80,pomerance89} to show that for some constant $\eta > 0$ and all large $x$, we have $\max_{d\le x}\#\T_{\CM}(d) > x^{\eta}$. 

\section{Complete determination	 of torsion in odd degrees:\\ Proof of Theorem \ref{thm:strongodddegree}}

\subsection{Background} Let $\ell >2$ be prime, and let $F$ be a number field of odd degree. If $E_{/F}$ is a CM elliptic curve with an $F$-rational point of order $\ell^n$, then the Odd Degree Theorem gives that $\ell \equiv 3 \pmod{4}$. Moreover, we also have the following additional results of \cite{BCS}. Here,  $\zeta_{\ell^n}$ denotes a primitive $\ell^{n}$th root of unity, and $h(\Delta)$ denotes the class number of the order of discriminant $\Delta$.

\begin{thm}[\cite{BCS}]
\label{GaloisImage}
Let $F$ be a number field of odd degree, and let $E_{/F}$ be a CM elliptic curve with a point of order $\ell^n$ in $E(F)$ for some prime $\ell \equiv 3 \pmod{4}$. Then:
\begin{enumerate}
\item E has CM by an order of discriminant $\Delta=-(2^{\epsilon}\ell^a)^2\cdot \ell$  for $\epsilon \in \{0,1\}$ and $a \in \Z_{\geq 0}$.
\item $\Q(\zeta_{\ell^n}) \subset FK$.
\item $\dfrac{\varphi(\ell^n)}{2}h(\Delta) \mid [F:\Q]$.
\item If $E$ has CM by an order of discriminant $\Delta=-\ell^{2a+1}$, there is a basis of $E[\ell^n]$ such that if $\sigma \in \Gal(\bar{F}/FK)$, then the image of $\sigma$ under the mod-$\ell^n$ Galois representation associated to $E$ is of the following form:
\[
 \rho_{\ell^n}(\sigma)=\left[ \begin{array}{cc} 1 &  \beta \left(\frac{\Delta-t^2}{4t}\right) \\ 0 & 1 \end{array} \right], \hspace{ .2cm} 4t \mid t^2-\Delta, \, \beta t \equiv 0 \pmod{\ell^n}.
\]

\end{enumerate}
\end{thm}

\begin{proof}
Since $h(\Delta)=[\Q(j(E)):\Q]$ and $\Q(j(E))$ is a subfield of $F$, it follows that $h(\Delta)$ is odd. Then part (1) is a consequence of the Odd Degree Theorem and \cite[Lemma 3.5]{BCS}. Parts (2) and (3) may be deduced from \cite[Theorem 4.12]{BCS}. Part (4) follows from the proof of \cite[Theorem 4.12(a)]{BCS}.
\end{proof}

As indicated in the introduction, elliptic curves with CM by a non-maximal order play a significant role in determining $\T_{\CM}(d)$ for odd $d$. One approach to analyzing rational torsion points on an elliptic curve with CM by a non-maximal order is to consider the rational torsion points on an elliptic curve with smaller conductor induced by the following natural isogeny. We thank Pete L. Clark for the idea of the next proof.

\begin{prop}
\label{isogeny}
Let $E_{/F}$ be an elliptic curve with CM by the order $\OO$ in $K$ of conductor $f$. If $f' \mid f$, there exists an $F$-rational isogeny $\iota_{f'}\colon E \rightarrow E'$, where $E'_{/F}$ is an elliptic curve with CM by the order in $K$ of conductor $f'$. Moreover, $\iota_{f'}$ is cyclic of degree $d=\frac{f}{f'}$.
\end{prop}

\begin{proof}
Let $\OO'$ be the order of $K$ of conductor $f'$, and let $I =d\OO'$. Then $I$ is an ideal of $\OO$, and we may form the $I$-torsion kernel,
\[
E[I]=\{x \in E(\bar{F}) : \forall \alpha \in I, \, \alpha x=0\}.\]
Since $I$ is fixed by complex conjugation, $E[I]$ is defined over $F$. (See the discussion in Section 3.3 of \cite{BCS}.) Thus we have an $F$-rational isogeny $E \rightarrow E/E[I]$. It remains to show $E/E[I]$ has CM by $\OO'$ and that $E[I]$ is cyclic of order $d$.

Choose an embedding $F \hookrightarrow \mathbb{C}$ so that $E \cong_{\mathbb{C}} E_{\OO}$, where $E_{\OO}$ is the $\OO$-CM elliptic curve corresponding to $\mathbb{C}/\OO$ under uniformization. Since $\OO \subset \OO'$, we have a natural map $\mathbb{C}/\OO \rightarrow \mathbb{C}/\OO'$, and hence a map from $E_{\OO} \rightarrow E_{\OO'}$. Then as in Proposition 1.4 of \cite{advanced},
\begin{align*}
\Ker(E_{\OO} \rightarrow E_{\OO'}) &\cong \Ker(\mathbb{C}/\OO \rightarrow \mathbb{C}/\OO')\\
&=\OO'/\OO\\
&=\{z \in \mathbb{C} : z \OO' \subset \OO'\}/\OO
\end{align*}

We will show $\{z \in \mathbb{C} : z \OO' \subset \OO'\}=\left\{z \in \mathbb{C} : z d \OO' \subset \OO\right\}$. Indeed, if $z\OO' \subset \OO'$, then in particular $z=a+f'k$ for some $a \in \Z$ and $k \in \Ok$. Then $(a+f'k)d\OO' \subset \OO$, as desired. Conversely, if $zd\OO' \subset \OO$, then $z=(a/d)+f'k$, for some $a \in \Z$ and $k \in\Ok$. Note that if $d$ divides $a$, then $z\OO' \subset \OO'$, as desired. But $d$ must divide $a$, for otherwise $zd\OO'$ would not be contained in $\OO$. Thus:
\begin{align*}
\Ker(\mathbb{C}/\OO \rightarrow \mathbb{C}/\OO') &=\left\{z \in \mathbb{C} : z d \OO' \subset \OO\right\}/\OO\\
&=\{z \in \mathbb{C} : \alpha z \in \OO \, \forall \, \alpha \in I\}/\OO\\
&=\{z \in \mathbb{C}/\OO : \alpha z = 0 \, \forall \, \alpha \in I\}\\
&=\mathbb{C}/\OO[I]\\
&\cong E_{\OO}[I].
\end{align*}

Composing with the isomorphism $E \cong_{\mathbb{C}} E_{\OO}$ gives an isogeny $E \rightarrow E_{\OO'}$ with kernel $E[I]$. It follows that $E/E[I] \cong_{\mathbb{C}} E_{\OO'}$, and so $E/E[I]$ has CM by $\OO'$. Moreover, $ E[I] \cong \OO'/\OO$, which is cyclic of order $d$.
\end{proof}
\begin{remark} 
If $f'=1$, we recover the classical statement which appears, for example, as Proposition 25 in \cite{TOR1}.
\end{remark}

Finally, we will make use of the connection between CM elliptic curves and class field theory, as stated in the following result. For a positive integer $N$ and an imaginary quadratic field $K$, we let $K^{(N\Ok)}$ denote the $N$-ray class field of $K$. 
\begin{prop}
\label{RayClass}
Let $F$ be a number field, and let $E_{/F}$ be an elliptic curve with CM by an order in $K$. If $(\Z/N\Z)^2 \hookrightarrow E(F)$ for some $N \in \Z^+$, then $K^{(N\Ok)} \subset FK$.
\end{prop}
\begin{proof}
For CM by the maximal order, see \cite[Theorem II.5.6]{advanced}. For the general case, see \cite[Theorem 5]{CP15}.
\end{proof}

\begin{lemma}\label{unity}
Let $\ell \equiv 3 \pmod{4}$ be prime and $K=\Q(\sqrt{-\ell})$. If $n \in \Z^+$, then for any $a \geq n$,
\[\Q(\zeta_{\ell^{a}}) \cap K^{\ell^n \Ok}=\Q(\zeta_{\ell^n}).\]
\end{lemma}

\begin{proof}
Let $\OO$ be the order in $K$ of discriminant $\Delta=-\ell^{2n+1}$, and let $K_{\OO}$ denote the ring class field of $K$ of conductor $\ell^n$. Then by Corollary 8.7 of Cox \cite{cox}, $K_{\OO} \subset K^{\ell^n \Ok}$. The field $\Q(\zeta_{\ell^{a}}) \cap K_{\OO}$ is generalized dihedral over $\Q$ by Theorem 9.18 of Cox \cite{cox}, and since it is also abelian over $\Q$ we have 
\[
\Gal(\Q(\zeta_{\ell^{a}}) \cap K_{\OO}/K) \cong (\Z/2\Z)^{\nu}.
\]
However, $[K_{\OO}:K]=h(\Delta)$ is odd by \cite[Lemma 3.5]{BCS}. Hence $\Q(\zeta_{\ell^{a}}) \cap K_{\OO}=K$.

Since $\Q(\zeta_{\ell^n}) \subset K^{\ell^n \Ok}$, we have $\Q(\zeta_{\ell^n})\subset \Q(\zeta_{\ell^{a}}) \cap K^{\ell^n \Ok}$. Let $\delta=[\Q(\zeta_{\ell^{a}}) \cap K^{\ell^n \Ok} : \Q(\zeta_{\ell^n})]$. Then the compositum of $K_{\OO}$ and $\Q(\zeta_{\ell^{a}}) \cap K^{\ell^n \Ok}$ has degree
\[
\frac{1}{w_K}h_K \ell^{2n-1}(\ell-1)\delta
\]
over $K$, where $w_K$ denotes the number of roots of unity in $K$. Since 
\[
[K^{\ell^n \Ok}:K]=\frac{1}{w_K}h_K \ell^{2n-1}(\ell-1)
\]
(see \cite[Corollary 3.2.5]{cohen2}), we find $\delta=1$ and $\Q(\zeta_{\ell^{a}}) \cap K^{\ell^n \Ok}=\Q(\zeta_{\ell^n})$. 
\end{proof}

\subsection{Identifying torsion in lowest degree.} We will use part (4) of Theorem \ref{GaloisImage} to deduce a relationship between the discriminant $\Delta$ and the rational torsion of the elliptic curve. Since $t \mid \Delta$, we find a connection between $\text{ord}_{\ell}(\Delta)$ and the full torsion over $FK$ forced by the existence of an $F$-rational point of order $\ell^n$.  For example, suppose $F$ is a number field of odd degree and $E_{/F}$ is an elliptic curve with CM by an order of discriminant $\Delta=-\ell$, where $\ell \equiv 3 \pmod{4}$ is prime. If $E(F)$ contains a point of order $\ell^n$, then $E$ has full $\ell^{n-1}$ torsion over $FK$ by Theorem \ref{GaloisImage}(4). This kind of argument is key to ruling out points of order $\ell^n$ appearing in low degree on elliptic curves with small conductor. Instead, we find that elliptic curves $E$ defined over $F=\Q(j(E))$ which possess a cyclic $F$-rational isogeny of degree $\ell^n$ give examples of rational points of order $\ell^n$ appearing in lowest possible odd degree. Such isogenies have been classified by Kwon \cite{kwon99}.

\begin{thm} \label{main}
Let $F$ be a number field of odd degree, let $\ell \equiv 3 \pmod{4}$ be prime, and let $n \in \Z^+$. If $E_{/F}$ is a CM elliptic curve with a point of order $\ell^n$ in $E(F)$, then
\[
h_{\Q(\sqrt{-\ell})} \cdot \frac{\ell-1}{2}\cdot \ell^{\delta} \mid [F:\Q],
\]
where
\[\delta = \begin{cases} 
      \left \lfloor{\frac{3n}{2}}\right \rfloor-1, & \ell > 3,\\
      0, & \ell=3 \text{ and } n=1,\\
      \left \lfloor{\frac{3n}{2}}\right \rfloor-2, & \ell=3 \text{ and } n \geq 2.
   \end{cases}\]
Moreover, for any such $\ell$ and $n$, there exists a CM elliptic curve defined over a number field of degree $h_{\Q(\sqrt{-\ell})} \cdot \frac{\ell-1}{2}\cdot \ell^{\delta}$ with a rational point of order $\ell^n$.
\end{thm}

\begin{proof}
Let $F$ be a number field of odd degree, and suppose $E_{/F}$ is a CM elliptic curve with a point $P$ of order $\ell^n$ in $E(F)$, where $\ell \equiv 3 \pmod{4}$ is prime. For now suppose $\ell \neq 3$. It follows from Theorem \ref{GaloisImage} that $E$ has CM by an order in $K=\Q(\sqrt{-\ell})$ of discriminant $\Delta=-(2^{\epsilon}\ell^a)^2\cdot \ell$ for $\epsilon \in \{0,1\}$ and $a \in \Z_{\geq 0}$, and
\begin{equation}
\label{div1}
\frac{\varphi(\ell^n)}{2} \cdot h(\Delta) =h_K \frac{\ell -1}{2} \ell^{a+n-1}\left(2-\left(\frac{-\ell}{2}\right)\right)^{\epsilon} \mid [F:\Q].
\end{equation} (The formula for $h(\Delta)$ appears in \cite[Theorem 7.24]{cox}.) If $a \geq  \left \lfloor{\frac{n}{2}}\right \rfloor$, then this quantity is divisible by $h_K \frac{\ell -1}{2} \ell^{\delta}$, as desired. So we may assume $a <  \left \lfloor{\frac{n}{2}}\right \rfloor$.

Let $\varphi\colon E \rightarrow E'$ be the $F$-rational isogeny of degree $2^{\epsilon}\ell^a$ whose existence is guaranteed by Proposition \ref{isogeny}, where $E'$ is an elliptic curve with CM by $\Ok$. Then $\varphi(P)$ has order $\ell^{\alpha}$, where $\alpha \geq n-a \geq 2$. Indeed, $n-a <2$ contradicts $a <  \left \lfloor{\frac{n}{2}}\right \rfloor$. By Theorem \ref{GaloisImage}, there is a basis of $E'[\ell^{\alpha}]$ such that if $\sigma \in \Gal(\bar{F}/FK)$, then $\rho_{\ell^{\alpha}}(\sigma)$ is of the following form:
\[
\left[ \begin{array}{cc} 1 &  \beta \left(\frac{-\ell-t^2}{4t}\right) \\ 0 & 1 \end{array} \right], \quad 4t \mid t^2+\ell, \quad \beta t \equiv 0 \pmod{\ell^{\alpha}}.
\]
In particular, $t \mid \ell$, so $\beta t \equiv 0 \pmod{\ell^{\alpha}}$ implies $\beta \equiv 0 \pmod{\ell^{\alpha-1}}$. Thus $E'$ has full $\ell^{\alpha-1}$-torsion over $FK$. Since $FK$ contains $\Q(\zeta_{\ell^n})$ by Theorem \ref{GaloisImage} and $K^{\ell^{\alpha-1}\Ok}$ by Proposition \ref{RayClass}, it follows from Lemma \ref{unity} that $h_K (\ell-1) \ell^{n+\alpha-2} \mid [FK:\Q]$. Hence 
\begin{equation}
\label{div2}
h_K \frac{(\ell-1)}{2} \ell^{n+\alpha-2} \mid [F:\Q].
\end{equation} 

\begin{figure}[H]
\begin{center}
\begin{tikzpicture}[node distance=2cm]
\node (Q)                  {$\mathbb{Q}$};
\node (K) [above of=Q, node distance=1cm] {$K$};
\node (Q1) [above of=K, node distance=1.5cm] {$\Q(\zeta_{\ell^{\alpha-1}})$};
\node (Kl) [above left of =Q1, node distance=2 cm] {$K^{(\ell^{\alpha-1} \Ok)}$};
\node (Qll)  [above right of=Q1, node distance=2.5 cm]   {$\Q(\zeta_{\ell^n})$};
\node (Comp) [above left of =Qll, node distance=1.8cm] {$K^{(\ell^{\alpha-1} \Ok)}(\zeta_{\ell^n})$};

 \draw[-] (Q1) edge node[right] {$\ell^{n-\alpha+1}$} (Qll);
 \draw[-] (Q) edge node[left] {2} (K);
 \draw[-] (Qll) edge node[right] {$h_K \ell^{\alpha-1}$} (Comp);
 \draw[-] (Q1) edge node[left] {$h_K \ell^{\alpha-1}$} (Kl);
 \draw[-] (Q1) edge node[left] {$\frac{1}{2}(\ell-1)\ell^{\alpha-2}$} (K);
 \draw (Kl) -- (Comp);

\end{tikzpicture}
\end{center}
\end{figure}

\noindent Since $a <  \left \lfloor{\frac{n}{2}}\right \rfloor$, we have $n+\alpha-2 \geq 2n-a-2 \geq \delta$. It follows that $h_K \frac{(\ell-1)}{2} \ell^{\delta} \mid [F:\Q]$.

If $\ell =3$, then by Theorem \ref{GaloisImage} $E$ has CM by an order $\OO$ in $K=\Q(\sqrt{-3})$ of discriminant $\Delta=-(2^{\epsilon}3^a)^2\cdot 3$ for $\epsilon \in \{0,1\}$ and $a \in \Z_{\geq 0}$. In addition, Theorem \ref{GaloisImage} implies
\begin{equation}
\label{div3}
\frac{\varphi(3^n)}{2} \cdot h(\Delta) =\frac{3^{a+n+\epsilon-1}}{[\Ok^{\times}:\OO^{\times}]}\mid [F:\Q].
\end{equation} If $a \geq  \left \lfloor{\frac{n}{2}}\right \rfloor$, then this quantity is divisible by $h_K \frac{3 -1}{2} 3^{\delta}=3^{\delta}$, as desired. So we may assume $a <  \left \lfloor{\frac{n}{2}}\right \rfloor$. Arguing as above, we find that $FK$ contains both $K^{3^{\alpha-1}\Ok}$ and $\mathbb{Q}(\zeta_{3^n})$ for some $\alpha \geq n-a \geq 2$. By Lemma \ref{unity}, these fields are linearly disjoint over $\mathbb{Q}(\zeta_{3^{\alpha-1}})$; hence, 
\begin{equation}
\label{div4}
3^{n+\alpha-3} \mid [F:\mathbb{Q}].
\end{equation} Since $a <  \left \lfloor{\frac{n}{2}}\right \rfloor$, we have $n+\alpha-3 \geq 2n-a-3 \geq \delta$. It follows that $3^{\delta} \mid [F:\Q]$.

It remains to show these divisibility conditions are best possible. Set $a =  \left \lfloor{\frac{n}{2}}\right \rfloor$. Then $\ell^n \mid \Delta=-(\ell^{a})^2\ell$. Let $E$ be an $\OO(\Delta)$-CM elliptic curve defined over $F=\Q(j(E))$. By work of Kwon \cite[Corollary 4.2]{kwon99}, $E$ admits an $F$-rational isogeny which is cyclic of degree $\ell^n$. It follows from \cite[Theorem 5.6]{BCS} that there is a twist $E_1$ of $E_{/F}$ and an extension $L/F$ of degree $\varphi(\ell^n)/2$ such that $E_1(L)$ has a point of order $\ell^n$. We have $[L:\Q]=\frac{\varphi(\ell^n)}{2}h(\Delta)=h_K\frac{\ell-1}{2}\ell^{\delta}$, as desired.
\end{proof}

\begin{remark}
\label{order}
If $F$ is a number field of odd degree and $E_{/F}$ is a CM elliptic curve with an $F$-rational point of order $\ell^n$, it follows from the proof of Theorem \ref{main} that $E$ has CM by an order of discriminant $\Delta=-(2^{\epsilon}\ell^a)^2\ell$ where $a= \left \lfloor{\frac{n}{2}}\right \rfloor$, $\left  \lfloor{\frac{n}{2}}\right \rfloor-1$, or $\left  \lfloor{\frac{n}{2}}\right \rfloor+1$. The latter two cases are only possible if $n$ is even or if $n=1$ (and $\ell=3$), respectively. In particular, it we see that $E$ necessarily has CM by a non-maximal order if $n \geq 3$.
\end{remark}

\begin{cor}
\label{field}
Let $\ell \equiv 3 \pmod{4}$ be prime, and let $n \in \Z^{+}$. Let $F$ be a number field of odd degree. If $E_{/F}$ is a CM elliptic curve with a point of order $\ell^n$ in $E(F)$, then $E$ has CM by $K=\Q(\sqrt{-\ell})$ and $FK$ contains $K^{\ell^{\left \lfloor{\frac{n}{2}}\right \rfloor}\Ok}(\zeta_{\ell^n})$.
\end{cor}

\begin{proof}
By Theorem \ref{GaloisImage}, $E$ has CM by the order $\OO$ in $K=\Q(\sqrt{-\ell})$ of discriminant $\Delta=-(2^{\epsilon}\ell^a)^2\cdot \ell$ for $\epsilon \in \{0,1\}$ and $a \in \Z_{\geq 0}$. We consider two cases.

If $a \geq \left \lfloor{\frac{n}{2}}\right \rfloor$, then $\OO$ is contained in the order $\OO'$ of conductor $\ell^{\left \lfloor{\frac{n}{2}}\right \rfloor}$. Thus the ring class field of $K$ with conductor $\ell^{\left \lfloor{\frac{n}{2}}\right \rfloor}$, $K_{\OO'}$, is contained in $K_{\OO}=K(j(E)) \subset FK$. (See exercise 9.19 of Cox \cite{cox}.) Since $\Q(\zeta_{\ell^n}) \subset FK$ by Theorem \ref{GaloisImage} and $\Q(\zeta_{\ell^n}) \cap K_{\OO'}=K$ by the proof of Lemma \ref{unity}, $K_{\OO'}(\zeta_{\ell^n})$ is a subfield of $FK$ of degree $h_{\Q(\sqrt{-\ell})} \cdot (\ell-1)\cdot \ell^{\delta}$. By Corollary 8.7 of Cox \cite{cox}, $K_{\OO'} \subset K^{\ell^{\left \lfloor{\frac{n}{2}}\right \rfloor}\Ok}$, so $K_{\OO'}(\zeta_{\ell^n}) \subset K^{\ell^{\left \lfloor{\frac{n}{2}}\right \rfloor}\Ok}(\zeta_{\ell^n})$. Since they have the same degree, equality holds. 

If $a < \left \lfloor{\frac{n}{2}}\right \rfloor$, then $FK$ contains $K^{\ell^{\alpha-1}\Ok}$ and $\Q(\zeta_{\ell^n})$, where $\alpha \geq n-a \geq \left \lfloor{\frac{n}{2}}\right \rfloor+1$. Thus $FK$ contains $K^{\ell^{\left \lfloor{\frac{n}{2}}\right \rfloor}\Ok}(\zeta_{\ell^n})$. \end{proof}

\begin{cor}
\label{conductor}
Let $\ell \equiv 3 \pmod{4}$ be prime, and let $n \in \Z^{+}$. Suppose $E_{/F}$ is a CM elliptic curve with a point of order $\ell^n$ in $E(F)$. If $[F:\Q]=h_{\Q(\sqrt{-\ell})} \cdot \frac{\ell-1}{2}\cdot \ell^{\delta}$, for $\delta$ defined as above, then $E$ has CM by an order in $K=\Q(\sqrt{-\ell})$ and $FK=K^{\ell^{\left \lfloor{\frac{n}{2}}\right \rfloor}\Ok}(\zeta_{\ell^n})$. In particular, $\ell$ is the only prime which ramifies in $F$.
\end{cor}

\begin{proof}
This follows from Corollary \ref{field}.
\end{proof}

\subsection{Proof of the Strong Odd Degree Theorem} Let $T$ be a group which appears as the torsion subgroup of a CM elliptic curve defined over a number field of odd degree. We will identify an odd positive integer $d_{T}$ such that $d_{T} \mid [F:\Q]$ whenever $E_{/F}$ is a CM elliptic curve with $E(F)[\tors] \cong T$ and $F$ is of odd degree. Once we exhibit a number field $F$ of degree $d_{T}$ and a CM elliptic curve $E_{/F}$ with $E(F)[\tors]\cong T$, the Strong Odd Degree Theorem will follow by a result of \cite{BCS}:

\begin{thm}
\label{Thm2.1}
Let $A_{/F}$ be an abelian variety over a number field, and let $d \geq 2$. There are infinitely many $L/F$ such that $[L:F]=d$ and $A(L)[\tors]=A(F)[\tors]$.
\end{thm}

\begin{proof}
See Theorem 2.1 of \cite{BCS}.
\end{proof}

We isolate the more involved case in the following lemma. Here, $\delta$ is as defined in the statement of the Strong Odd Degree Theorem.

\begin{lemma}
\label{2TorDiv}
Let $F$ be a number field of odd degree. If $E_{/F}$ is a CM elliptic curve with $E(F)[\tors] \cong \Z/2\ell^n\Z$ for $\ell \equiv 3 \pmod{8}$, where $n \geq 2$ if $\ell=3$, then $3 \cdot h_{\Q(\sqrt{-\ell})} \cdot \frac{\ell-1}{2}\cdot \ell^{\delta} \mid [F:\Q]$.
\end{lemma}

\begin{proof}
We will first consider the case where $\ell \equiv 3 \pmod{8},\, \ell \neq 3$. Suppose $E_{/F}$ is an elliptic curve with CM by an order $\OO$ of discriminant $\Delta$ in $K$, and suppose $E(F)[\tors] \cong \Z/2\ell^n \Z$. By Theorem \ref{GaloisImage}, $\Delta=-(2^{\epsilon}\ell^a)^2\cdot \ell$ for $\epsilon \in \{0,1\}$ and $a \in \Z_{\geq 0}$. If $\epsilon=0$, then $K^{2\Ok} \subset F(\sqrt{d})=FK$ by \cite[Lemma 3.15]{BCS} and Proposition \ref{RayClass}. By Corollary \ref{field}, $K^{\ell^{\left \lfloor{\frac{n}{2}}\right \rfloor}\Ok}(\zeta_{\ell^n}) \subset FK$. Since 2 ramifies in $K^{2\Ok}$ and $K^{\ell^{\left \lfloor{\frac{n}{2}}\right \rfloor}\Ok}(\zeta_{\ell^n})$ is unramified away from $\ell$, these fields are linearly disjoint over $K^{\Ok}$. Thus $3\cdot h_K(\ell-1)\ell^{\delta} \mid [FK:\Q]$, and $3 \cdot h_K \cdot \frac{\ell-1}{2}\cdot \ell^{\delta} \mid [F:\Q]$. If $\epsilon=1$, then the ring class field of $K$ of conductor $2$ is contained in $K_{\OO}=K(j(E))=FK$ (see exercise 9.19 of Cox \cite{cox}). Since this ring class field and $K^{\ell^{\left \lfloor{\frac{n}{2}}\right \rfloor}\Ok}(\zeta_{\ell^n})$ are linearly disjoint over $K^{\Ok}$, it follows that $3 \cdot h_{\Q(\sqrt{-\ell})} \cdot \frac{\ell-1}{2}\cdot \ell^{\delta} \mid [F:\Q]$. 

Suppose $\ell=3$, and suppose $E_{/F}$ has CM by an order $\OO$ of discriminant $\Delta$ in $K$ and $E(F)[\tors] \cong \Z/2\cdot 3^n \Z$. As in the lemma statement, we assume $n \geq 2$. We first consider the case where $2 \mid \Delta$, i.e., $\Delta=-(2\cdot3^a)^2\cdot 3$ for $a \in \Z_{\geq 0}$. If $a \geq \left \lfloor{\frac{n}{2}}\right \rfloor$, then $3^{\delta+1} \mid [F:\Q]$ by equation \eqref{div3}. If $a < \left \lfloor{\frac{n}{2}}\right \rfloor$, then we must consider several sub-cases:
\begin{itemize}
\item $a \geq 1$: Since $6$ divides the conductor of $\OO$, the ring class field of conductor $6$, $K_{\OO'}$, is contained in $K_{\OO}=K(j(E)) \subset FK$ (see exercise 9.19 of Cox \cite{cox}). The prime 2 ramifies in $K_{\OO'}$, so $K^{3^{\left \lfloor{\frac{n}{2}}\right \rfloor}\Ok}(\zeta_{3^n})$ and $K_{\OO'}$ are linearly disjoint over $K$. Since $[K_{\OO'}:K]=3$, we have $3^{\delta+1} \mid [F:\Q]$.
\item $a=0$, $n \geq 3$: The proof of Theorem \ref{main} shows that a rational point of order $2 \cdot 3^n$ forces $K^{3^{\alpha-1}\Ok}(\zeta_{3^n}) \subset FK$, where $\alpha \geq n$. Thus $K^{3^{n-1}\Ok}(\zeta_{3^n}) \subset FK$, which means $3^{2n-3} \mid [F:\Q]$. Since $n \geq 3$, we have $3^{\delta+1} \mid [F:\Q]$.
\item $a=0$, $n=2$: Let $P$ be the point of order 18 in $E(F)[\tors]$, where $E$ is an elliptic curve with CM by an order of discriminant $\Delta=-2^2 \cdot 3$. Note $j(E)=2^43^35^3$. Work of \cite{Ku76} implies $E$ has an equation of the form
\[
y^2+(1-c)xy-by=x^3-bx^2
\]
for some $b,c \in F$ and $P=(0,0)$. We let $j(b,c)$ denote the $j$-invariant of $E$. As in \cite{tor2}, we may obtain a polynomial $f_{18} \in \Q[b,c]$ that vanishes when $(0,0)$ has order 18. A computation shows that if 
\[
\begin{cases}f_{18}(b,c)=0 \\ j(b,c)=2^43^35^3
\end{cases},
\]
then $9 \mid [\Q(b,c):\Q]$ (see the research website of the first author for the \texttt{Magma} scripts used). Hence $9 \mid [F:\Q]$, as desired.
\end{itemize}

Next, suppose $2 \nmid \Delta$, i.e., $\Delta=-(3^a)^2\cdot 3$ for $a \in \Z_{\geq 0}$. If $a \geq 1$, then by Proposition \ref{isogeny}, there exists an $F$-rational isogeny $\iota_3\colon E \rightarrow E'$, where $E'$ has CM by the order in $K=\Q(\sqrt{-3})$ of conductor 3. Since the kernel of this isogeny has size $3^{a-1}$, a point of order 2 in $E(F)$ induces a point of order 2 in $E'(F)$. But $E'$ is a quadratic twist of the elliptic curve $E_0: y^2=x^3-480x+4048$, and points of order 2 are invariant under quadratic twists. Thus $E_0(F)$ contains a point of order 2, and $F$ contains a root $\alpha$ of $x^3-480x+4080$. Since 2 ramifies in $K(\alpha)$, the fields $K(\alpha)$ and $K^{3^{\left \lfloor{\frac{n}{2}}\right \rfloor}\Ok}(\zeta_{3^n})$ are linearly disjoint over $K$; hence $3^{\delta+1} \mid [F:\Q]$.

If $a=0$ and $n \geq 3$, $K^{3^{n-1}\Ok}(\zeta_{3^n}) \subset FK$ and $3^{\delta+1} \mid [F:\Q]$ as above. Finally, if $a=0$ and $n=2$, then $j(E)=0$. As in the case where $2 \mid \Delta$, a computation shows that $9 \mid [F:\Q]$. Again, see the research website of the first author for the \texttt{Magma} scripts used.
\end{proof}

 \begin{thm}[Strong Odd Degree Theorem]
 Let $\ell \equiv 3 \pmod{4}$ and $n \in \Z^+$. Define $\delta$ as follows:
 \[\delta = \begin{cases} 
      \left \lfloor{\frac{3n}{2}}\right \rfloor-1, & \ell > 3,\\
      0, & \ell=3 \text{ and } n=1,\\
      \left \lfloor{\frac{3n}{2}}\right \rfloor-2, & \ell=3 \text{ and } n \geq 2.
   \end{cases}\]\\
Then:
 \begin{enumerate}
  \item For any odd positive integer $d$, the groups $\{ \bullet\}, \, \Z/2\Z, \, \Z/4\Z,$ and $\Z/2\Z \oplus \Z/2\Z$ appear as the torsion subgroup of a CM elliptic curve defined over a number field of degree $d$.
\item $\Z/\ell^n\Z$ appears as the torsion subgroup of a CM elliptic curve defined over a number field of odd degree $d$ if and only if $\ell \equiv 3 \pmod{8}$ and $d$ is a multiple of $h_{\Q(\sqrt{-\ell})} \cdot \frac{\ell-1}{2}\cdot \ell^{\delta}$.
 \item $\Z/2\ell^n\Z$ appears as the torsion subgroup of a CM elliptic curve defined over a number field of odd degree $d$ if and only if one of the following holds:
 \begin{enumerate}
 \item[a.] $\ell \equiv 3 \pmod{8}$, where $n \geq 2$ if $\ell=3$, and $d$ is a multiple of $3 \cdot h_{\Q(\sqrt{-\ell})} \cdot \frac{\ell-1}{2}\cdot \ell^{\delta}$, or
 \item[b.] $\ell=3$ and $n=1$ and $d$ is any odd positive integer, or
 \item[c.] $\ell \equiv 7 \pmod{8}$ and $d$ is a multiple of $h_{\Q(\sqrt{-\ell})} \cdot \frac{\ell-1}{2}\cdot \ell^{\delta}$.
 \end{enumerate}
\end{enumerate}
\end{thm}

\begin{proof} The groups $\{ \bullet\}, \, \Z/2\Z, \, \Z/4\Z,$ and $\Z/2\Z \oplus \Z/2\Z$ appear as torsion subgroups of CM elliptic curves defined over $\Q$ by work of Olson \cite{Olson74}, so part (1) is an immediate consequence of Theorem \ref{Thm2.1}. For part (2), suppose $E_{/F}$ is a CM elliptic curve with $E(F)[\tors] \cong \Z/\ell^n\Z$. Then $\ell \equiv 3 \pmod{8}$ by the Odd Degree Theorem, and $[F:\Q]$ is a multiple of $h_{\Q(\sqrt{-\ell})} \cdot \frac{\ell-1}{2}\cdot \ell^{\delta}$ by Theorem \ref{main}. Conversely, if for each $\ell \equiv 3 \pmod{8}$ there exists a number field $F$ of degree $h_{\Q(\sqrt{-\ell})} \cdot \frac{\ell-1}{2}\cdot \ell^{\delta}$ and a CM elliptic curve $E_{/F}$ with $E(F)[\tors] \cong \Z/\ell^n\Z$, part (2) will follow from Theorem \ref{Thm2.1}.

First suppose $\ell \equiv 3 \pmod{8}$, $\ell \neq 3$. By Theorem \ref{main} there exists a CM elliptic curve $E$ defined over a number field $F$ of degree $h_{\Q(\sqrt{-\ell})} \cdot \frac{\ell-1}{2}\cdot \ell^{\delta}$ with $\ell^n \parallel E(F)[\tors]$. By the Odd Degree Theorem, $E$ has CM by $K=\Q(\sqrt{-\ell})$ and $E(F)[\tors] \cong \Z/\ell^n\Z$ or $\Z/2\ell^n\Z$. But if $E(F)[\tors] \cong \Z/2 \ell^n\Z$, then $3 \cdot h_K \cdot \frac{\ell-1}{2}\cdot \ell^{\delta} \mid [F:\Q]$ by Lemma \ref{2TorDiv}, which is a contradiction. Thus $E(F)[\tors] \cong \Z/\ell^n\Z$, as desired. If $\ell =3$, we know that $\Z/3\Z$ occurs in degree 1 by work of Olson \cite{Olson74}, and $\Z/9\Z$ occurs in degree 3 by work of Clark, Corn, Rice, and Stankewicz \cite{tor2}. If $n \geq 3$, we know there is an elliptic curve $E$ defined over a number field of $F$ degree $3^{\delta}$ with $3^n \parallel E(F)[\tors]$ by Theorem \ref{main}. If $E(F)[\tors] \cong \Z/2 \cdot 3^n\Z$, then $3 ^{\delta+1} \mid [F:\Q]$ by Lemma \ref{2TorDiv}. Thus the Odd Degree Theorem guarantees $E(F)[\tors] \cong \Z/3^n\Z$, as desired. This completes the proof of part 2.

Let $F$ be a number field of odd degree, and suppose $E_{/F}$ is a CM elliptic curve with $E(F)[\tors] \cong \Z/2\ell^n\Z$ for some prime $\ell \equiv 3 \pmod{4}$. If $\ell \equiv 3 \pmod{8}$, where $n \geq 2$ if $\ell=3$, then $3 \cdot h_{\Q(\sqrt{-\ell})} \cdot \frac{\ell-1}{2}\cdot \ell^{\delta} \mid [F:\Q]$ by Lemma \ref{2TorDiv}. If $\ell \equiv 7 \pmod{8}$, then $h_{\Q(\sqrt{-\ell})} \cdot \frac{\ell-1}{2}\cdot \ell^{\delta} \mid [F:\Q]$ by Theorem \ref{main}. Thus part 3 will follow from Theorem \ref{Thm2.1} if we can demonstrate that there is a CM elliptic curve $E$ defined over a number field $F$ of smallest possible odd degree with $E(F)[\tors] \cong  \Z/2\ell^n\Z$.

Suppose $\ell \equiv 3 \pmod{8}$, where $n \geq 2$ if $\ell=3$. By the proof of part 2, there is a number field $F$ of degree $h_K \cdot \frac{\ell-1}{2}\cdot \ell^{\delta}$ and a CM elliptic curve $E_{/F}$ with $E(F)[\tors] \cong \Z/\ell^n\Z$. Since points of order 2 correspond to the roots of a cubic polynomial, $E$ gains a rational $2$-torsion point over a cubic extension of $F$, say $F(\alpha)$. By the Odd Degree Theorem and Lemma \ref{2TorDiv}, $E(F(\alpha))[\tors] \cong \Z/2\ell^n\Z$. Since $[F:\Q]=3 \cdot h_{\Q(\sqrt{-\ell})} \cdot \frac{\ell-1}{2}\cdot \ell^{\delta}$, we may conclude part 3(a).

If $\ell=3$ and $n=1$, then $\Z/2\ell\Z$ does occur in degree 1 by Olson \cite{Olson74}. Thus 3(b) holds. For 3(c), let $\ell \equiv 7 \pmod{8}$. By Theorem \ref{main} there exists a CM elliptic curve $E$ defined over a number field $F$ of degree $h_{\Q(\sqrt{-\ell})} \cdot \frac{\ell-1}{2}\cdot \ell^{\delta}$ with $\ell^n \parallel E(F)[\tors]$. The Odd Degree Theorem shows $E(F)[\tors] \cong \Z/2\ell^n\Z$, as desired.
\end{proof}

\section{The upper order of $T_{\CM}(d)$ for odd degrees $d$:\\Proof of Theorem \ref{thm:upperodd}}

Here we exploit the fact that while $h_{\Q(\sqrt{-\ell})}$ is typically of size $\asymp \ell^{1/2}$, it can be smaller by a factor of size $1/\log\log{\ell}$, but (assuming GRH) no more. The precise statements we need correspond via Dirichlet's class number formula to the following two estimates.

\begin{prop}[Joshi]\label{prop:joshi} There is a sequence of primes $\ell \equiv 3\pmod{4}$, $\ell\to\infty$, with
\[ L(1,{\textstyle\leg{-\ell}{\cdot}}) \le \frac{\pi^2}{6e^{\gamma}} \frac{1}{\log\log{\ell}}.\]
\end{prop}
\begin{proof} This is part of \cite[Theorem 1]{joshi70}. 
\end{proof}

\begin{prop}[Littlewood]\label{prop:littlewood} Assume the Riemann Hypothesis for Dirichlet $L$-functions. Then as $|D|\to\infty$, with $D$ ranging through fundamental discriminants,
\[ L(1,{\textstyle{\leg{D}{\cdot}}}) \ge \left(\frac{\pi^2}{12e^{\gamma}} +o(1)\right)\frac{1}{\log\log{|D|}}. \]
For $D$ satisfying $\leg{D}{2}=1$, this lower bound can be strengthened to
\[ L(1,{\textstyle{\leg{D}{\cdot}}}) \ge \left(\frac{\pi^2}{4e^{\gamma}} +o(1)\right)\frac{1}{\log\log{|D|}}. \]
\end{prop}
\begin{proof} The first assertion is explicitly contained in \cite[Theorem 1]{littlewood28}. The second can be proved by the same method. For the sake of completeness, we sketch an argument  for the second claim taking as a starting point the the modern approach to Littlewood's work presented in \cite[\S5.2]{LLS15}. From \cite[eq. (5.2)]{LLS15}, we see that with $X=\frac{1}{4}(\log{|D|})^2$,
\[ \log L(1,{\textstyle\leg{D}{\cdot}}) \ge \sum_{n \le X}\Lambda(n) {\textstyle\leg{D}{n}} \left(\frac{1}{n\log{n}} -\frac{1}{X\log{X}}\right) + o(1), \]
as $|D|\to\infty$. (We have dropped lower order terms from \cite{LLS15}, since we are only interested in asymptotics, not in an explicit bound.) For each prime $p$, the contribution to the right-hand sum from $n$ that are powers of $p$ is at least
\[ \sum_{p^k \le X} \Lambda(p^k) (-1)^k\left(\frac{1}{p^k\log{p^k}} -\frac{1}{X\log{X}}\right); \]
moreover, since $\leg{D}{2}=1$, the factor $(-1)^k$ appearing here can be replaced with $1$ when $p=2$. Now summing over $p$,
\[ \log L(1,{\textstyle\leg{D}{\cdot}}) \ge \sum_{p^k \le X} \Lambda(p^k) (-1)^k\left(\frac{1}{p^k\log{p^k}} -\frac{1}{X\log{X}}\right) + 2\sum_{\substack{2^k \le X \\ k \text{ odd}}} \Lambda(2^k) \left(\frac{1}{2^k\log{2^k}}-\frac{1}{X\log{X}}\right) + o(1).\]
As on p. 2408 of \cite{LLS15}, 
\[ \sum_{p^k \le X} \Lambda(p^k) (-1)^k\left(\frac{1}{p^k\log{p^k}} -\frac{1}{X\log{X}}\right) \ge -\log\log{X}-\gamma +\log\frac{\pi^2}{6} + o(1).\]
Moreover,
\begin{multline*} 2\sum_{\substack{2^k \le X \\ k \text{ odd}}} \Lambda(2^k) \left(\frac{1}{2^k\log{2^k}}-\frac{1}{X\log{X}}\right) = 2 \sum_{\substack{2^k \le X \\ k \text{ odd}}} \frac{1}{k \cdot 2^k} + o(1) \\ = 2 \sum_{k \ge 1} \frac{1}{k \cdot 2^k} - \sum_{j \ge 1}\frac{1}{j \cdot 2^{2j}} + o(1)= 2\ln\left(\frac{1}{1-\frac{1}{2}}\right) + \ln\left(1-\frac{1}{4}\right) + o(1) = \ln(3) + o(1). 
\end{multline*}
Collecting the estimates and exponentiating (and noting that $\log{X} \sim 2\log\log{|D|})$, we obtain the claim.
\end{proof}

We can now prove the unconditional lower-bound half of Theorem \ref{thm:upperodd}.

\begin{proof}[Proof of Theorem \ref{thm:upperodd}(i)] We fix a sequence of primes $\ell$ as in Proposition \ref{prop:joshi}. To each such $\ell$, we associate the odd positive integer $d = h_{\Q(\sqrt{-\ell})} \cdot \frac{\ell-1}{2}$. Clearly, $d\to\infty$ as $\ell\to\infty$. By Dirichlet's class number formula,
\[ d = \frac{\sqrt{\ell}}{\pi}L(1,{\textstyle{\leg{-\ell}{\cdot}}}) \cdot \frac{\ell-1}{2} \le \left(\frac{\pi}{12e^{\gamma}} + o(1)\right) \ell^{3/2}/\log\log{\ell}, \]
as $\ell\to\infty$. It is straightforward to deduce that $\ell^{3/2} \ge (\frac{12e^{\gamma}}{\pi}+o(1)) d\log\log{d}$. From the Strong Odd Degree Theorem, either $\Z/\ell\Z$ or $\Z/2\ell\Z$ is realizable in degree $d$, and so $T_{\CM}(d) \ge \ell$. Theorem \ref{thm:upperodd}(i) follows.
\end{proof}

The proof of the upper bound is more intricate.


\begin{proof}[Proof of Theorem \ref{thm:upperodd}(ii)] We will assume to start with that the Riemann Hypothesis for $L$-functions holds, and we will prove that under this assumption,
\begin{equation}\label{eq:RHupper} T_{\CM}(d) \le \left(\left(\frac{24e^{\gamma}}{\pi}+o(1)\right)d\log\log{d}\right)^{2/3}\end{equation}
as $d\to\infty$ through odd values. We say a few words at the end about how to modify the proof to obtain the unconditional upper bound $T_{\CM}(d) \ll_{\epsilon} d^{2/3+\epsilon}$.

From the Odd Degree Theorem, the largest torsion subgroup realizable in degree $d$ has the form $\Z/\ell^n \Z$ or $\Z/2\ell^n\Z$ for a prime $\ell \equiv 3\pmod{4}$ and a positive integer $n$. Here the prime $\ell$ and the positive integer $n$ are uniquely determined by $d$. 

\begin{pcases}\case[$n$\text{ is even}] From Theorem \ref{thm:strongodddegree} along with the bounds $\frac{\ell-1}{2} \ge \frac{\ell}{3}$ and $h_{\Q(\sqrt{-\ell})} \ge 1$, 
\[ d \ge \frac{\ell-1}{2} \cdot \ell^{\delta} \cdot h_{\Q(\sqrt{-\ell})} \ge \frac{1}{3}\ell^{\delta+1} \ge \frac{1}{9} \ell^{3n/2}. \]
To see the last estimate, notice that $\ell^{\delta+1} = \ell^{3n/2}$ unless $\ell=3$, in which case it is $\frac{1}{3}\ell^{3n/2}$. Hence, $\ell^n \le (9d)^{2/3}$ and $T_{\CM}(d) \le 2\ell^{n} \le 2\cdot (9d)^{2/3}$. This certainly implies \eqref{eq:RHupper} for these $d$. 

\case[$\ell < \log\log{d}$] Here Theorem \ref{thm:strongodddegree} implies that
\[ d\ge \frac{1}{3}\ell^{\delta+\frac{3}{2}} \cdot \frac{h_{\Q(\sqrt{-\ell})}}{\ell^{1/2}} \ge \frac{1}{9}\ell^{3n/2} \cdot \frac{h_{\Q(\sqrt{-\ell})}}{\ell^{1/2}} \ge \frac{1}{9 (\log\log{d})^{1/2}} \ell^{3n/2}. \]
Thus, $T_{\CM}(d) \le 2\ell^{n} \ll d^{2/3} (\log\log{d})^{1/3}$ in these cases, so \eqref{eq:RHupper} again holds.
 
\case[$n$ odd and $\ell \ge \log\log{d}$] Taking $D=-\ell$ in Proposition \ref{prop:littlewood} and invoking the class number formula, we deduce from Proposition \ref{prop:littlewood} that as $d\to\infty$,
\begin{align} d \ge h_{\Q(\sqrt{-\ell})} \cdot \frac{\ell-1}{2} \cdot \ell^{\delta} &=
\frac{\sqrt{\ell}}{\pi} L(1,{\textstyle\leg{-\ell}{\cdot}}) \cdot \frac{\ell-1}{2} \ell^{\delta}\label{eq:n1case0}\\
&\ge \left(\frac{\pi}{24e^{\gamma}}+o(1)\right) \frac{\ell^{\delta+3/2}}{\log\log{\ell}} =  \left(\frac{\pi}{24e^{\gamma}}+o(1)\right) \frac{\ell^{3n/2}}{\log\log{\ell}}. \label{eq:n1case}
\end{align}
This certainly implies that $d\ge \ell$ for large $d$, and hence $\log\log{d} \ge \log\log{\ell}$. Feeding this back into the above estimate gives
\[ d \ge  \left(\frac{\pi}{24e^{\gamma}}+o(1)\right) \frac{\ell^{3n/2}}{\log\log{d}},\quad\text{whence}\quad \ell^{n} \le \left(\left(\frac{24e^{\gamma}}{\pi} + o(1)\right)d\log\log{d}\right)^{2/3}. \] As a consequence, if the largest torsion subgroup in degree $d$ has the form $\Z/\ell^n\Z$, rather than $\Z/2\ell^n\Z$, then we again obtain \eqref{eq:RHupper}. It remains to treat the subcase when the largest torsion subgroup has the form $\Z/2\ell^n\Z$. Recall that $d\to\infty$ and $\ell \ge \log\log{d}$, so certainly we can assume $\ell > 3$. If $\ell \equiv 3\pmod{8}$, Theorem \ref{thm:strongodddegree} shows that the first inequality in \eqref{eq:n1case0} can be strengthened by a factor of $3$. If $\ell \equiv 7\pmod{8}$, then $\leg{-\ell}{2}=1$, and Proposition \ref{prop:littlewood} shows that the inequality in \eqref{eq:n1case} can be strengthened by a factor of $3$. Following the argument through shows that in either case,
\[ \ell^n \le \left(\left(\frac{8e^{\gamma}}{\pi} + o(1)\right)d\log\log{d}\right)^{2/3}, \]
and hence
\[ T_{\CM}(d) \le 2\left(\left(\frac{8e^{\gamma}}{\pi} + o(1)\right)d\log\log{d}\right)^{2/3}. \]
Since $8 \cdot 2^{3/2} < 24$, \eqref{eq:RHupper} holds in this case as well.
\end{pcases}

The proof of the unconditional bound is similar but simpler. The key difference is that in the treatment of odd $n$, we are forced to use Siegel's lower bound $L(1,\leg{D}{\cdot}) \gg_{\epsilon} |D|^{-\epsilon}$ instead of the much stronger results of Proposition \ref{prop:littlewood}.
\end{proof}

\section{The density of Olson degrees:\\ Proof of Theorem \ref{thm:olsondensity}}
\subsection{Upper bound} \label{sec:upper}
To bound the density of Olson degrees from above, we must bound from below the density of $\M(\Gg)$ for the set $\Gg$ appearing in Proposition \ref{prop:BCP}. There is an obvious plan of attack: Bound $\M(\Gg)$ from below by $\M(\Hh)$ for a large finite subset $\Hh \subset \Gg$. For example, since $g_{5} = 3$ and $g_{11}=5$, we have $\{2,3,5\} \subset \Gg$, and so
\[ \d(\M(\Gg)) \ge \d(\M(\{2,3,5\})) = \frac{11}{15}. \]
This implies the upper bound of $\frac{4}{15}$ --- mentioned in the introduction --- for the density of Olson degrees. In this section, we implement the same strategy with a much larger set $\Hh$.

It requires some finesse to make this method computationally feasible. For any finite set $\Hh$ of positive integers, inclusion-exclusion immediately yields a formula for $\M(\Hh)$ , namely
\[ \d(\M(\Hh)) = \sum_{j=1}^{\#\Hh} (-1)^{j-1} \sum_{\substack{\sA \subset \Hh \\ \#\sA = j}} \frac{1}{\mathrm{lcm}(\sA)}. \]
Unfortunately, the above formula involves $2^{\#\Hh}-1$ terms and so a direct implementation of this idea quickly becomes prohibitively time-consuming. To work around this we make two observations, encoded in the following lemmas.
\begin{lemma}\label{lem:elem1} For a finite collection $\Hh$ of positive integers, let \[ \Hh_{\mathrm{rel}}=\{h\in \Hh: h \text{ is relatively prime to all other elements of $\Hh$}\}. \] Then
\[ 1-\d(\M(\Hh)) = \left(1-\d(\M(\Hh\setminus \Hh_{\mathrm{rel}}))\right) \prod_{h \in \Hh_{\mathrm{rel}}}\left(1-\frac{1}{h}\right).\]
\end{lemma}

If $\Hh$ is a finite set of natural numbers and $p$ is a prime, we define \emph{the $p$-scaled set} $\Hh_{(p)}$ by
\[ \Hh_{(p)} = \{h/\gcd(h,p): h \in \Hh\}, \]
and we define the \emph{$p$-sieved set} $\Hh^{(p)}$ by
\[ \Hh^{(p)} = \{h \in \Hh: p\nmid h\}. \]

\begin{lemma}\label{lem:elem2} Let $\Hh$ be a finite collection of positive integers. For any prime number $p$,
\[ \d(\M(\Hh)) = \frac{1}{p} \d(\M(\Hh_{(p)})) + \left(1-\frac{1}{p}\right) \d(\M(\Hh^{(p)})). \]
 \end{lemma}

\noindent Lemmas \ref{lem:elem1} and \ref{lem:elem2} make pleasant elementary exercises, and we omit the proofs. The more difficult of the two, Lemma \ref{lem:elem2}, appears in more general form in work of Behrend \cite[Lemma, p. 681]{behrend48}.

\begin{proof}[Proof of the upper bound in Theorem \ref{thm:olsondensity}] We begin by computing a large list of elements of $\Gg$ which eventually will be truncated to form our $\Hh$. Specifically, we start with the singleton set $\{2\}$. We then successively go through the primes $3 < \ell \le 100000$ from the congruence class $3\bmod{4}$, throwing $g_{\ell}$ into our set whenever  $g_{\ell}$ is not divisible by a preexisting element. (If $g_{\ell}$ is divisible by such an element, then there is no need to throw it in, as this would not lead to a larger set of multiples.) At the end of this process, we sort the resulting list; this leaves us with a set the first several elements of which are
\begin{multline*} 2, 3, 5, 913, 1631, 1703, 2051, 2891, 3247, 3401, 3619, 4067, 5327, 6251, 6617, 7051, 7183, 7429, 9737, \\ 10829, 11129, 11143, 12389, 12463, 12673, 12847, 17611, 18403, 19253, 19931, 20033, 22211, 22747, \\ 23351, 27491, 28237, 30173, 32927, 33541, 38171, 38641, 39311, 39689, 40687, 42601, 45103, \dots. \end{multline*}
We let $\Hh$ consist of the first 38 elements of this list, so that
\[ \Hh= \{2, 3, 5, 913, \dots, 32927\}. \]We will show that
\[ 1-\d(\M(\Hh)) < 0.265.\]
This implies the same upper bound $0.265$ for the density of Olson degrees.

Apply Lemma \ref{lem:elem1} to $\Hh$. In this case, one computes that $\Hh_{\mathrm{rel}}=\{2,3,5,11129,27491\}$. Thus, puting $\Hh'=\Hh\setminus \Hh_{\mathrm{rel}}$,
\[ 1-\d(\M(\Hh)) =  (1-\d(\M(\Hh'))) \bigg(1-\frac{1}{2}\bigg) \bigg(1-\frac{1}{3}\bigg) \bigg(1-\frac{1}{5}\bigg) \bigg(1-\frac{1}{11129}\bigg) \bigg(1-\frac{1}{27491}\bigg).\]
To estimate $\d(\M(\Hh'))$, we apply Lemma \ref{lem:elem2} with $p=11$:
\[ \d(\M(\Hh')) = \frac{1}{11} \d(\M(\Hh'_{(11)})) + \left(1-\frac{1}{11}\right) \d(\M(\Hh'^{(11)})). \]
The set $\Hh'^{(11)}$ has only $23$ elements, and so $\d(\M(\Hh'^{(11)}))$ can be computed without fuss by inclusion-exclusion. We find that
\[ \d(\M(\Hh'^{(11)})) = 0.004217267361708\dots.\]
Let $\Hh''=\Hh'_{(11)}$. Then $\Hh''$ has $33$ elements; $33$ is large enough that a direct inclusion-exclusion computation is best avoided. So we make another application of Lemma \ref{lem:elem1}, this time to $\Hh''$. We compute that $\Hh''_{\mathrm{rel}}= \{641,653,1013,1133,1601\}$. So with $\Hh'''=\Hh''\setminus\Hh''_{\mathrm{rel}}$,
\[ 1-\d(\M(\Hh'')) = (1-\d(\M(\Hh''')) \bigg(1-\frac{1}{641}\bigg) \bigg(1-\frac{1}{653}\bigg) \bigg(1-\frac{1}{1013}\bigg) \bigg(1-\frac{1}{1133}\bigg) \bigg(1-\frac{1}{1601}\bigg). \]
The set $\Hh'''$ has $28$ elements. However, it contains both $83$ and $4067 = 83 \cdot 49$. So we may remove $4067$ from $\Hh'''$ without changing the corresponding set of multiples. Similarly, $\Hh'''$ contains both $329$ and $6251= 19 \cdot 329$, and so $6251$ can also be removed. This brings $\#\Hh'''$ down to $26$, which is small enough that the inclusion-exclusion computation is manageable. We find that
\[ 1-\d(\M(\Hh''')) = 0.979914305743609\dots.\]
Working back through the chain of equalities,
\[ 1-\d(\M(\Hh'')) = 0.974452539520107\dots,\]
\[ \d(\M(\Hh')) = 0.006156375826997\dots, \]
and finally
\[ 1-\d(\M(\Hh)) = 0.264991512979231\dots. \]
This completes the proof of the upper bound half of Theorem \ref{thm:olsondensity}.\end{proof}

\subsection{Lower bound}  For the rest of this section, $\ell$ always denotes a prime with $\ell> 3$ and $\ell \equiv 3\bmod{4}$.  
To establish the lower bound in Theorem \ref{thm:olsondensity}, we require a lower bound on the numbers $g_{\ell}$ that holds ``most of the time''. Via Dirichlet's class number formula, this comes down to bounding below $L(1,\leg{-\ell}{\cdot})$. We will deduce what we need from the following variant of the Siegel--Tatuzawa theorem, due to Chen \cite{chen07}.
\begin{prop}\label{prop:ST} Let $0 < \epsilon < \frac{1}{\log(10^6)}$. For all real primitive characters $\chi$ of conductor $q > \exp(1/\epsilon)$, with at most one exception, 
 \begin{equation}\label{eq:STlower} L(1,\chi) > \min\left\{\frac{1}{7.732\log{q}}, \frac{1.5\cdot 10^6 \cdot \epsilon}{q^{\epsilon}}\right\}. \end{equation}
\end{prop}
\noindent Let $\epsilon_0 = \frac{0.999}{\log(10^6)}$, and apply Proposition \ref{prop:ST} with $\epsilon = \epsilon_0$. The minimum in \eqref{eq:STlower} corresponds to the first term when $q \lessapprox 2.82 \cdot 10^{115}$, and to the second term past this point. Moreover, for $q> 10^{115}$, one checks that the right-hand side of \eqref{eq:STlower} is bounded below by $10^5 \cdot q^{-\epsilon_0}$. We use these observations in the proof of the next result.

\begin{cor}\label{cor:classno} For all negative fundamental discriminants $D$ with $|D|>10^{6}$, except for a single possible exception, we have
\[ h_{\Q(\sqrt{D})} > 0.041 \sqrt{|D|}/\log{|D|} \quad\text{when}\quad |D| \le 10^{115}, \]
and, with $\epsilon_0 = 0.999/\log(10^6)$,
\[ h_{\Q(\sqrt{D})} > 3 \cdot 10^{4}\cdot |D|^{\frac{1}{2}-\epsilon_0} \quad\text{when}\quad |D| > 10^{115}.\]
\end{cor}
\begin{proof} Recall that when $D$ is negative and $|D| > 4$, we have $h_{\Q(\sqrt{D})} = \frac{\sqrt{|D|}}{\pi} L(1,\textstyle{\leg{D}{\cdot}})$. Since $\frac{1}{7.732\pi} = 0.0411\dots$ and $\frac{10^5}{\pi} > 3 \cdot 10^4$, the result follows.
\end{proof}

\begin{proof}[Proof of the lower bound in Theorem \ref{thm:olsondensity}] We have already noted that $\Gg \supset \{2,3,5\}$ and that
\[ \d(\M(\{2,3,5\})) = \frac{11}{15} = 0.7333\ldots.\]
The lower bound claimed in Theorem \ref{thm:olsondensity} is equivalent to the assertion that $\M(\Gg)$ has density $< 0.736$. So it is enough to show that, with $\overline{\d}(\cdot)$ denoting upper density,
\begin{equation}\label{eq:goal} \overline{\d}(\M(\Gg)\setminus \M(\{2,3,5\})) < 0.0026. \end{equation}

Suppose that $m \in \M(\Gg)$ and $m \notin \M(\{2,3,5\})$. Then $m$ has the form $g_{\ell} r$, where  $\gcd(g_{\ell},30)=\gcd(r,30)=1$. Fixing $\ell$ with $\gcd(g_{\ell},30)=1$, the number of corresponding $m \le x$ is $\frac{4}{15}\frac{x}{g_{\ell}}  + O(1)$. (The $O(1)$ error term comes from the application of inclusion-exclusion to enforce the condition $\gcd(r,30)=1$.) Hence, the total number of such $m \le x$ is at most \[ \frac{4}{15} x \sum_{\substack{\ell:~g_{\ell} \le x \\ \gcd(g_{\ell},30)=1}} \frac{1}{g_{\ell}} + O(x/\log{x}).\] Dividing by $x$ and letting $x\to\infty$ shows that
\[ \overline{\d}(\M(\Gg)\setminus \M(\{2,3,5\})) \le \frac{4}{15} \sum_{\ell:~\gcd(g_{\ell},30)=1} \frac{1}{g_{\ell}}. \]
We write
\[ \sum_{{\ell:~\gcd(g_{\ell},30)=1}} \frac{1}{g_{\ell}} = \sum\nolimits_1 + \sum\nolimits_2 + \sum\nolimits_3, \]
where $\sum\nolimits_1$, $\sum\nolimits_2$, and $\sum\nolimits_3$ indicate a restriction to the ranges $\ell\le 10^{9}$, $10^9 < \ell \le 2.8 \cdot 10^{9}$, and $\ell > 2.8 \cdot 10^{9}$, respectively.

We treat these three sums in turn. The first sum can be calculated directly in \texttt{PARI}, using the routine \texttt{quadclassunit} to compute the class numbers of the fields $\Q(\sqrt{-\ell})$. We find that
\[ \sum\nolimits_1 < 0.00788. \]
We remark that, in general, \texttt{PARI}'s function \texttt{quadclassunit} is only guaranteed to produce correct output assuming the truth of GRH. However, in our range of $\ell$, the GRH-conditional result employed here has been verified by extensive computations of Jacobson, Ramachandran, and Williams. (See the discussion  in \cite[\S3.4]{JRW06}.) So our estimation of $\sum_1$ is in fact unconditional. To treat $\sum_2$, we recall that Watkins \cite{watkins04} has shown that $h_{\Q(\sqrt{D})} > 100$ for all negative fundamental discriminants with $|D| > 2383747$. Thus, replacing the condition $\gcd(g_{\ell},30)=1$ by the weaker hypothesis that $\gcd(\frac{\ell-1}{2},30)=1$,
\[ \sum\nolimits_2 < \sum_{\substack{10^9 < \ell \le 2.8\cdot 10^{9} \\ \gcd(\frac{\ell-1}{2},30)=1}} \frac{1}{\frac{\ell-1}{2} \cdot 100} < 0.0001819,  \]
where again the final estimate comes from an explicit computation in \texttt{PARI}.

It remains to treat $\sum_3$. We write $\sum_3 = \sum_3' + \sum_3''$, where $'$ is the contribution of the possible exceptional $D=-\ell$ described in Corollary \ref{cor:classno}, and $''$ is the contribution from all other $\ell$. Using that $h_{\Q(\sqrt{-\ell})} > 100$ for this exceptional $\ell$ (if it exists),
\[ \sum\nolimits_3' < \frac{1}{\frac{2.8\cdot 10^9-1}{2} \cdot 100} < 10^{-11}. \]

We turn next to $\sum_3''$. Let $\Pi(t)$ be the number of $\ell \in (3,t]$ with $\gcd(\frac{\ell-1}{2},30)=1$. Each $\ell$ counted here satisfies $\ell\equiv 3\pmod{4}$, $\ell\equiv 2\pmod{3}$, and $\ell\equiv 2,3,$ or $4\pmod{5}$. So $\ell$ is forced into $3$ of the $\varphi(60)=16$ reduced residue classes modulo $60$. By the Brun--Titchmarsh theorem in the explicit form of Montgomery--Vaughan \cite{MV73},
\[ \Pi(t) \le 2\frac{3}{16} \frac{t}{\log{(t/60)}} = \frac{3}{8} \frac{t}{\log(t/60)}\]
for every $t> 60$.

Applying Corollary \ref{cor:classno}, we find that
\begin{align*} \sum\nolimits_3'' &< \sum_{\substack{ 2.8\cdot 10^{9} < \ell \le 10^{115}\\ \gcd(\frac{\ell-1}{2},30)=1}}\frac{1}{\frac{\ell-1}{2} \cdot \frac{0.041 \sqrt{\ell}}{\log{\ell}}} + \sum_{\substack{\ell > 10^{115}\\ \gcd(\frac{\ell-1}{2},30)=1}}\frac{1}{\frac{\ell-1}{2} \cdot 3 \cdot 10^4 \cdot \ell^{1/2-\epsilon_0}}\\&= \frac{2}{0.041} \int_{2.8\cdot 10^9}^{10^{115}} \frac{\log{t}}{(t-1)\sqrt{t}}\,d\Pi(t) + \frac{2}{3\cdot 10^4} \int_{10^{115}}^{\infty} \frac{1}{(t-1) \cdot t^{1/2-\epsilon_0}}\, d\Pi(t) \\
&< \frac{2}{0.041} \int_{2.8\cdot 10^9}^{\infty} \Pi(t) \left(-\frac{\log{t}}{(t-1)\sqrt{t}}\right)'\, dt + \frac{2}{3\cdot 10^4}\int_{10^{115}}^{\infty} \Pi(t) \left(-\frac{1}{(t-1) t^{1/2-\epsilon_0}}\right)'\, dt. \end{align*}
Inserting the above upper bound for $\Pi(t)$ and using \texttt{Mathematica} to bound the resulting integrals from above, we find that
\[ \sum\nolimits_3'' < 0.001220. \]
Putting everything together,
\begin{align*} \overline{\d}(\M(\Gg)\setminus \M(\{2,3,5\})) &\le \frac{4}{15}\left(\sum\nolimits_1 + \sum\nolimits_2 + \sum\nolimits_3\right) \\ &< \frac{4}{15}\left(0.00788+0.0001819+10^{-11}+ 0.001220\right) < 0.00248.\end{align*}
This establishes \eqref{eq:goal} and so completes the proof of the lower bound in Theorem \ref{thm:olsondensity}.\end{proof}

\section{Stratification of torsion in odd degrees:\\Proof of Theorem \ref{thm:stratification}}

We begin with an analytic lemma concerning integers with prescribed sets of divisors. Let $\Gg$ be a set of positive integers. For each positive integer $n$, put $\D(n,\Gg) = \{g\in \Gg: g \mid n\}$. 

\begin{lemma}\label{lem:prescribeddivisors} Let $\Gg$ be a set of odd positive integers, and suppose that the sum of the reciprocals of the elements of $\Gg$ converges. Let $\Hh$ be any finite subset of $\Gg$. The set of odd $n$ with $\D(n,\Gg)= \Hh$ possesses a well-defined asymptotic density; this density is positive as long as there is at least one such $n$.
\end{lemma}

\begin{proof} We prove the lemma in two steps. First, we show that the density exists, and then we show positivity. Let $\sA$ be the set of odd $n$ with $\D(n,\Gg)=\Hh$. For each real $z > \max\Hh$, put
\[ \sA_{z} = \{n: \D(n,\Gg \cap [1,z]) = \Hh\}. \]
Notice that whenever $z' > z > \max\Hh$,
\begin{equation}\label{eq:doubleinclusion} \sA \subset \sA_{z'} \subset \sA_{z}. \end{equation}
Now whether or not $n$ belongs to $\sA_{z}$ depends only on $n$ modulo $2\prod_{g \in \Gg \cap[1,z]}g$. Thus, $\sA_{z}$ is a finite union of congruence classes, and so $\d(\sA_{z})$ exists. From \eqref{eq:doubleinclusion}, $\d(\sA_{z})$ is a nonincreasing function of $z$, and so we we may define
\[ \delta= \lim_{z\to\infty} \d(\sA_{z}). \]
We will show that $\sA$ has asymptotic density $\delta$. 

In what follows, we continue to use $\overline{\d}(\cdot)$ for upper density, and we use $\underline{\d}(\cdot)$ for lower density.

From the first inclusion in \eqref{eq:doubleinclusion}, $\overline{\d}(\sA) \le \d(\sA_{z})$ for all $z$; now letting $z\to\infty$ shows that $\sA$ has upper density at most $\delta$. Now consider the lower density of $\sA$. If $n \in \sA_{z}$ but $n\notin \sA$, then $n$ is divisible by some $g \in \Gg$ with $g > z$; the number of these $n \le x$ is at most $x\sum_{g \in \Gg,~g>z} 1/g$. Dividing by $x$ and letting $x\to\infty$, it follows that
\[ \underline{\d}(\sA) \ge \d(\sA_{z}) - \sum_{g \in \Gg,~g>z}\frac{1}{g}. \]
Letting $z\to\infty$, and recalling our assumption that the reciprocal sum of the elements of $\Gg$ converges, we find that $\underline{\d}(\sA) \ge \delta$. Thus, $\d(\sA)=\delta$.

It remains to show the positivity of $\delta$ under the assumption that $\sA$ is nonempty. We prove this by exhibiting a subset of $\sA$ of positive lower density. Fix $n_0 \in \sA$. For each real $z > \max\Hh$, put
\[ M_z:= 2\prod_{g \in \Gg,~g\le z} g.\]
We consider $n$ of the form $n_0 m$, where $m\equiv 1\pmod{M_z}$. Clearly, the set of $m\equiv 1\pmod{M_z}$ has density $\frac{1}{M_z}$. Moreover, each number of the form $n_0 m$ is odd and satisfies $\D(n_0 m,\Gg \cap [1,z]) = \Hh$. So if $m \le x$ and $n_0 m\notin \sA$, then $g \mid n_0 m$ for some $g\in \Gg$ with $g \in (z,n_0 x]$;  hence,
\[ g/\gcd(g,n_0)\mid m.\]
Since $m\equiv 1\pmod{M_z}$, we must have $g/\gcd(g,n_0)$ coprime to $M_z$. Hence, the above divisibility forces $m$ into a uniquely determined residue class modulo $gM_z/\gcd(g,n_0)$. The number of these $m \le x$ is at most
\begin{align*} \sum_{\substack{z < g \le n_0 x \\ g \in \Gg,~(g/\gcd(g,n_0),M_z)=1}} \bigg(\frac{x \gcd(g,n_0)}{g M_z} + 1\bigg) &\le \frac{x}{M_z} \cdot \bigg(n_0 \sum_{\substack{g> z \\ g\in \Gg}}\frac{1}{g}\bigg) + \sum_{\substack{g \le n_0 x\\ g\in \Gg}} 1.\end{align*}
Observe that
\[ \sum_{\substack{g \le n_0 x \\ g\in \Gg}} 1 \le x^{1/2} + \sum_{\substack{x^{1/2} < g \le n_0 x\\ g\in \Gg}} 1 \le x^{1/2} + \sum_{\substack{x^{1/2} < g \le n_0 x\\ g\in \Gg}} \frac{n_0 x}{g} \le x^{1/2} + x\bigg(n_0 \sum_{\substack{g > x^{1/2}\\ g\in \Gg}}\frac{1}{g}\bigg). \]
Thus, the number of $m \le x$ with $m\equiv 1\pmod{M_z}$ and $n_0 m\notin \sA$ is at most
\[ \frac{x}{M_z} \cdot \bigg(n_0 \sum_{\substack{g> z \\ g\in \Gg}}\frac{1}{g}\bigg)+ x\bigg(n_0 \sum_{\substack{g > x^{1/2} \\ g\in \Gg}}\frac{1}{g}\bigg) + x^{1/2}.\]
Dividing by $x$ and letting $x\to\infty$, we find that the upper density of $m\equiv 1\pmod{M_z}$ with $n_0 m\notin \sA$ is at most
\[ \frac{1}{M_z} \left(n_0 \sum_{g> z}\frac{1}{g}\right). \] 
We now fix $z$ large enough that the parenthesized term is smaller than $1/2$. Then the lower density of $m$ with $n_0 m \in \sA$ is at least $\frac{1}{2M_z}$, and so the lower density of $\sA$ is at least $\frac{1}{2 n_0 M_z}$.
\end{proof}

We can now prove the first assertion of Theorem \ref{thm:stratification}.

\begin{proof}[Proof that the density of the $d$-Olson degrees exists and is positive, for odd $d$] Let $G$ run over the groups realizable as torsion subgroups in odd degree, as specified by the Odd Degree Theorem. For each such $G$, Theorem \ref{thm:strongodddegree} shows that $G$ is realizable in a particular odd degree $d$ precisely when a certain explicitly described positive odd integer $g_G$ (say) divides $d$. So with $\Gg= \{g_G: G\text{ realizable in odd degree}\}$, the $d$-Olson degrees are precisely the odd positive integers $d'$ with
\[ \D(d',\Gg) = \D(d,\Gg).\]
The existence of the density of $d$-Olson degrees, together with its positivity, now follows from Lemma \ref{lem:prescribeddivisors} once it is checked that $\sum_{g \in \Gg} 1/g$ converges. From Theorem \ref{thm:strongodddegree}, every $g \in \Gg$ has the form $h_{\Q(\sqrt{-\ell})} \cdot \frac{\ell-1}{2}\cdot \ell^{\delta}$ or $3h_{\Q(\sqrt{-\ell})} \cdot \frac{\ell-1}{2}\cdot \ell^{\delta}$ for some prime $\ell \equiv 3\pmod{4}$ and some nonnegative integer $\delta$. Hence,
\[ \sum_{g \in \Gg} \frac{1}{g} \le \frac{4}{3} \sum_{\ell} \frac{1}{h_{\Q(\sqrt{-\ell})} \cdot \frac{\ell-1}{2}} \sum_{\delta} \frac{1}{\ell^{\delta}} < \frac{8}{3} \sum_{\ell} \frac{1}{h_{\Q(\sqrt{-\ell})} \cdot \frac{\ell-1}{2}}. \]
As already observed in \cite{BCP}, this final sum on $\ell$ converges; for example, this follows from Siegel's lower bound $h_{\Q(\sqrt{-\ell})} \gg_{\epsilon} \ell^{1/2-\epsilon}$.
\end{proof}

It remains to prove that the densities of the sets of $d$-Olson degrees, for inequivalent odd $d$, sum to $1/2$. For this we need the following result from \cite{BCP}.

\begin{prop}[``Typical boundedness'' of torsion in the CM case] For each $\epsilon > 0$, there is a positive real number $z$ such that the set of (odd or even) $d$ with $T_{\CM}(d) > z$ has upper density smaller than $\epsilon$.
\end{prop}

\begin{proof}[Proof of the final assertion of Theorem \ref{thm:stratification}] We must show that for any complete set $\sD$ of inequivalent odd degrees, $\sum_{d \in\sD} \d(\{d\text{-Olson degrees}\}) = 1/2$. To begin, fix $\epsilon > 0$ and choose $z$ so that the integers $d$ with $T_{\CM}(d) > z$ comprise a set of upper density smaller than $\epsilon$. 

Since equivalent integers $d$ share the same value of $T_{\CM}(d)$, the set of odd $d$ with $T_{\CM}(d) \le z$ is a union of equivalence classes. Moreover, since there are only finitely many abelian groups of order at most $z$, this union is necessarily a finite one. So we can pick $d_1, \dots, d_k \in \sD$ with
\[ \{\text{odd d}: T_{\CM}(d) \le z\} = \bigcup_{i=1}^{k} \{d_i\text{-Olson degrees}\},\]
where the union on the right is disjoint. Exploiting finite additivity,
\begin{align*} \sum_{d \in \sD} \d(\{d\text{-Olson degrees}\}) &\ge \sum_{i=1}^{k} \d(\{d_i\text{-Olson degrees}\})\\ &= 
\d(\{\text{odd d}: T_{\CM}(d) \le z\}) = \frac{1}{2}- \d(\{\text{odd d}: T_{\CM}(d) > z\}) > \frac{1}{2}-\epsilon.\end{align*}

On the other hand, we also have that for each positive $Z$,
\[ \sum_{d \in \sD,~d\le Z} \d(\{d\text{-Olson degrees}\}) = \d(\bigcup_{d \in \sD,~d\le Z}\{d\text{-Olson degrees}\}) \le \d(\{\text{odd d}\}) = \frac{1}{2}. \]
Letting $Z\to\infty$, 
\[ \sum_{d \in \sD} \d(\{d\text{-Olson degrees}\}) \le \frac{1}{2}.\]

Since $\epsilon > 0$ is arbitrary, we conclude that $\sum_{d \in \sD} \d(\{d\text{-Olson degrees}\}) =1/2$.
\end{proof}

\section{The number of groups that can appear in a given degree}\label{sec:groupcounts}
\subsection{Odd degrees: Proof of Theorem \ref{thm:maxgroupcount}}
The proof of the lower bound in Theorem \ref{thm:maxgroupcount} will depend on the following ``Brun--Titchmarsh theorem on average'', which gives nontrivial information about primes in $[2,X]$ in arithmetic progressions with moduli slightly larger than $X^{1/2}$ (i.e., slightly beyond the range of applicability of the Bombieri--Vinogradov Theorem or the GRH). For the rest of this section, we fix the constant
\[ \delta_0 = 10^{-100}. \]
As usual, $\pi(x;q,a)$ denotes the count of primes $p\le x$ with $p\equiv a\pmod{q}$.

\begin{prop}\label{prop:BT} Let $A> 0$. If $X > X_0(A)$ and $Q \in [X^{1/2}, X^{1/2+\delta_0}]$, then 
\[ 0.85 \frac{X}{\varphi(q) \log{X}} \le \pi(X;q,1) \le 1.48 \frac{X}{\varphi(q) \log{X}} \]
for all $q \in [Q,2Q]$ except those belonging to an exceptional set $\E_A(X,Q)$ of cardinality not exceeding $Q (\log{X})^{-A}$.
\end{prop}
\begin{proof} This is a theorem of Rousselet \cite{rousselet88}.
\end{proof}

\begin{proof}[Proof of Theorem \ref{thm:maxgroupcount}] The upper bound is relatively straightforward. The Odd Degree Theorem implies that apart from $\{ \bullet\}, \, \Z/2\Z, \, \Z/4\Z,$ and $\Z/2\Z \oplus \Z/2\Z$, the elements of $\T(d)$ are of the form $\Z/\ell^n\Z$ or $\Z/2\ell^n\Z$, where $\ell \equiv 3\pmod{4}$ is prime. From Theorem \ref{thm:strongodddegree}, for $\Z/\ell^n\Z$ or $\Z/2\ell^n\Z$ to appear, it is necessary that $\frac{\ell-1}{2}$ divide $d$. Thus, the number of possible $\ell$ is at most $\tau(d)$. Theorem \ref{thm:strongodddegree} also implies that $n \le \frac{2}{3}v_{\ell}(d)+O(1)$, where $v_{\ell}(\cdot)$ is the $\ell$-adic valuation. Since $v_{\ell}(d) \ll \log(d)$, given $\ell$ there are only $O(\log(3d))$ possibilities for $n$. Hence, \[ \#\T_{\CM}(d) \ll 1 + \tau(d) \log(3d) \ll_{\epsilon} d^{\epsilon},\] where we use in the last step the well-known upper estimate for the maximal order of $\tau(d)$.

The lower bound requires significantly more effort. Recall from Theorem \ref{thm:strongodddegree} that if $\ell \equiv 3\pmod{4}$ is prime and $\frac{\ell-1}{2} h_{\Q(\sqrt{-\ell})} \mid d$, then at least one of $\Z/\ell\Z$ or $\Z/2\ell\Z$ belongs to $\T(d)$. So if $r(d)$ denotes the number of divisors of $d$ of the form $\frac{\ell-1}{2} h_{\Q(\sqrt{-\ell})}$, with $\ell$ as above, then
\[ \#\T(d) \ge r(d).\]
Now let $A$ be any large, fixed positive real number. We will show that there are infinitely many odd $d$ 
with 
\[ r(d) > (\log{d})^{\frac{1}{4} A \delta_0}. \] The lower bound in Theorem \ref{thm:maxgroupcount} is then immediate.

 In what follows, we allow implied constants to depend on $A$, and $\ell$ is understood to run only over primes from the residue class $3\bmod{4}$.

For each real number $x\ge 3$, put
\[ M:= \prod_{2 < p \le \frac{1}{2}\log{x}} p. \]
By the prime number theorem, $M \le x^{2/3}$ for large $x$. Our plan is to show that the average of $r(d)$ is large when taken over those $d \le x$ that are multiples of $M$. (This strategy is inspired by a similar argument of Prachar \cite{prachar55}. Cf. the proof of \cite[Proposition 10]{APR83}.) Hence, some individual term $r(d)$ must also be large. Now
\begin{align*} \sum_{\substack{d \le x \\ M \mid d}} r(d) &= \sum_{\substack{d \le x \\ M \mid d}} \#\{(m,\ell): m \left(\frac{\ell-1}{2}\right) h_{\Q(\sqrt{-\ell})} = d\} \\ &= \#\{(m,\ell): M \mid m\left(\frac{\ell-1}{2}\right)h_{\Q(\sqrt{-\ell})},\text{ and } m\left(\frac{\ell-1}{2}\right)h_{\Q(\sqrt{-\ell})} \le x\}. \end{align*}
We partition the pairs $(m,\ell)$ counted above according to the value of $\gcd(2M,m)$. Since we seek only a lower bound on the partial sums of $r(d)$, it is enough to consider pairs with $\gcd(2M,m)$ highly restricted. Let \[ T = (\log{x})^A. \] Given $g \mid 2M$ with $g \in (T,2T]$, we construct pairs $(m,\ell)$ with $\gcd(2M,m)=2M/g$ as follows: First, fix $\ell \le T^{2-\delta_0}$ with $g/\gcd(g,2) \mid \ell-1$. Choose any $m \le \frac{x}{h_{\Q(\sqrt{-\ell})} \cdot (\ell-1)/2}$ with $\gcd(2M,m)=2M/g$. Then the pair $(m,\ell)$ is counted above. Given $\ell$, inclusion-exclusion shows that the number of corresponding $m$ is (as $x\to\infty$)
\[ \sim \frac{x}{\frac{\ell-1}{2} h_{\Q(\sqrt{-\ell})}} \frac{1}{2M/g} \frac{\varphi(g)}{g}.  \]
Using $h_{\Q(\sqrt{-\ell})} \ll \ell^{1/2} \log{\ell}$ and $\varphi(g) \gg g/\log\log{g}$, the preceding expression is
\[ \gg \frac{x}{M (\log\log{x})^2} \cdot \frac{g}{\ell^{3/2}}.\]
Summing on $\ell$, and recalling that $\ell \le T^{2-\delta_0}$, the number of pairs we construct for our given $g$ is
\[ \gg \frac{x}{M (\log\log{x})^2} \frac{T}{T^{\frac{3}{2}(2-\delta_0)}} \sum_{\substack{\ell \le T^{2-\delta_0} \\ \ell \equiv 3\pmod{4} \\ \ell \equiv 1\pmod{g/\gcd(g,2)}}} 1.  \]
Now
\[ \sum_{\substack{\ell \le T^{2-\delta_0} \\ \ell \equiv 3\pmod{4} \\ \ell \equiv 1\pmod{g/\gcd(g,2)}}} 1 = \pi(T^{2-\delta_0}; g/\gcd(g,2), 1)- \pi(T^{2-\delta_0}; 4g/\gcd(g,2), 1). \]
In the notation of Proposition \ref{prop:BT},
\[ \pi(T^{2-\delta_0}; g/\gcd(g,2), 1) \ge 0.85 \cdot \frac{T^{2-\delta_0}}{\varphi(g/\gcd(g,2)) \log(T^{2-\delta_0})} \]
unless $g$ is even and $g/2 \in \E_1(T^{2-\delta_0}, T/2)$ or $g$ is odd and $g \in \E_1(T^{2-\delta_0},T)$; from Proposition \ref{prop:BT}, the number of $g$ involved in these exceptional sets is $O(T/\log\log{x})$. Similarly, as long as $g$ avoids a certain set of size $O(T/\log\log{x})$, we have
\[ \pi(T^{2-\delta_0}; 4g/\gcd(g,2), 1) \le 1.48 \cdot \frac{T^{2-\delta_0}}{2\varphi(g/\gcd(g,2)) \log(T^{2-\delta_0})}. \]
Inserting these prime counting estimates above (noting that $1.48/2 < 0.85$) we find that as long as $g$ avoids a set of size $O(T/\log\log{x})$, we construct
\[ \gg \frac{x}{M (\log\log{x})^2} \frac{T}{T^{\frac{1}{2}(2-\delta_0)}} \frac{1}{g \log{T}} \gg \frac{x}{M (\log\log{x})^3} \frac{1}{T^{1-\frac{1}{2}\delta_0}} \]
pairs for a given $g$. 

We now wish to sum over allowable values of $g$. Recall that our $g \in (T,2T]$ must satisfy $g\mid 2M$, i.e., $g$ must be a squarefree, $\frac{1}{2}\log{x}$-smooth number. Without the squarefree restriction, the number of these $g \le 2T$ is $\sim 2\rho(A) T$, where $\rho(\cdot)$ is Dickman's function; insisting that $g$ is squarefree cuts this down to $\sim \frac{12}{\pi^2} \rho(A) T$ (see, e.g., \cite{naimi88}). Similarly, the number of $\frac{1}{2}\log{x}$-smooth $g \le T$ is $\sim \frac{6}{\pi^2} \rho(A) T$. Thus, there are  $\gg T$ values of $g \in (T,2T]$ that divide $2M$. (This could also be established by more elementary methods.) Excluding the $O(T/\log\log{x})$ bad values of $g$ coming from the application of Proposition \ref{prop:BT}, we are still left with $\gg T$ allowable values of $g$ (for large $x$). It follows that
\[  \sum_{\substack{d \le x \\ M \mid d}} r(d) \gg \frac{x}{M (\log\log{x})^3} \frac{1}{T^{1-\frac{1}{2}\delta_0}} \cdot T = \frac{x}{M} \cdot \frac{(\log{x})^{\frac{1}{2} A \delta_0}}{(\log\log{x})^3}, \]
and hence
\[  \sum_{\substack{d \le x \\ M \mid d}} r(d) > \frac{x}{M} (\log{x})^{\frac{1}{4} A \delta_0}\]
for large enough $x$. 

Since there are no more than $x/M$ multiples of $M$ in $[1,x]$, there is some $d \le x$ satisfying $r(d) > (\log{x})^{\frac{1}{4} A \delta_0}$, and hence also $r(d) > (\log{d})^{\frac{1}{4} A \delta_0}$. Letting $x\to\infty$, we obtain infinitely many distinct $d$ of this kind.
\end{proof}

\subsection{Even degrees} In this section we explain why the upper bound in Theorem \ref{thm:maxgroupcount} fails if we do not restrict to odd values of $d$. The key ingredient in our argument is a variant of the following 1935 theorem of Erd\H{o}s \cite{erdos35} asserting the existence of ``popular'' values of Euler's $\varphi$-function.

\begin{prop}\label{prop:erdosprop} For some constant $\delta > 0$ and all sufficiently large $x$, 
\[ \max_{m \le x} \#\{n\text{ squarefree}: \varphi(n) = m\} > x^{\delta}. \]
\end{prop}

An easy modification of Erd\H{o}s's proof yields the following more general result.

\begin{prop}\label{prop:erdosgeneral} Let $\Pp$ be any subset of the primes with positive relative lower density. There are constants $\delta=\delta(\Pp)>0$ and $x_0=x_0(\Pp)$ such that for all $x> x_0$,
\[ \max_{m \le x} \#\{n:\text{n is a squarefree product of primes in $\Pp$},\text{ and } \varphi(n) = m\} > x^{\delta}. \]
\end{prop}
\begin{proof}[Proof (sketch)] We refer the reader to the proof of Proposition \ref{prop:erdosprop} appearing in the survey article \cite{pomerance89} (see that paper's Theorem 4.6). It is clear from that exposition that the proposition will follow if it is shown that, for some fixed $\alpha > 0$ and all large $T$, 
\begin{equation}\label{eq:smoothbound} \#\{p \le T: p \in \Pp,~\text{all prime factors of $p-1$ are at most $T^{1-\alpha}$}\} \gg T/\log{T}. \end{equation}
By hypothesis, there is a constant $c > 0$ such that for all large $T$,
\[ \#\{p \le T: p \in \Pp\} \ge cT/\log{T}. \]
From Brun's sieve, if $\alpha$ is fixed sufficiently close to $1$ and $T$ is large, then
\[ \#\{\text{primes }p \le T: q\mid p-1\text{ for some prime $q> T^{1-\alpha}$}\} < \frac{c}{2}T/\log{T}; \]
up to changes in notation, this assertion is contained in Erd\H{o}s's proof of \cite[Lemma 4]{erdos35}. Combining the last two estimates yields \eqref{eq:smoothbound}.
\end{proof}

\begin{thm}\label{thm:largeeven} For some constant $\eta > 0$ and all sufficiently large $x$,
\[ \max_{d \le x} \#\T_{\CM}(d) > x^{\eta}.\]
\end{thm}
\begin{proof} Let $E$ be an elliptic curve with $j(E)=0$, and choose a model of $E$ defined over $\Q$. Since $E$ has CM by the full ring of integers in $K=\Q(\sqrt{-3})$, the image of $\Gal(\bar{\Q}/K)$ under the mod-$\ell$ Galois representation associated to $E$ lands in a split Cartan subgroup for each prime $\ell\equiv 1\pmod{3}$. (See \cite[p. 12-13]{TOR1}.) Thus we have a $K$-rational cyclic subgroup of order $\ell$. Let $n=\ell_1\ell_2 \cdots \ell_r$, where $\ell_i \equiv 1\pmod{3}$ are distinct primes, and let $P_i$ be a generator of the $K$-rational subgroup of order $\ell_i$. Then $P_1+ \dots + P_r$ generates a $K$-rational subgroup of order $n$. By \cite[Theorem 5.5]{BCS}, there is a number field $F_n$ of degree dividing $\varphi(n)$ and a quadratic twist $E'$ of $E_{/F_n}$ such that $E'(F_n)$ has a point of order $n$. Below we write $E' = E'_{n}$ to indicate the dependence on $n$. By enlarging $F_n$ if necessary, we can (and will) assume that $[F_n:\Q]=\varphi(n)$.

Let $\Pp$ be the set of primes congruent to $1$ modulo $3$, so that $\Pp$ has relative density $1/2$. Let $\delta = \delta(\Pp)$ be the positive constant whose existence is specified in Proposition \ref{prop:erdosgeneral}. Thus, for all large $x$, there is an integer $d\le x$ for which the set 
\[ \N:=\{\text{squarefree numbers $n$ composed of primes $\ell\equiv 1\pmod{3}$}: \varphi(n) = d\}\]
satisfies $\#\N > x^{\delta}$. 
For each $n \in \N$, let $F_n$ be the number field specified in the previous paragraph, so that $[F_n:\Q]=d$. We will bound $\#\T_{\CM}(d)$ from below by showing that there is not too much repetition among the groups $E'_{n}(F_n)[{\rm tors}]$, for $n \in \N$.

Suppose that $n$ and $n'$ both belong to $\N$ and that $E'_{n}(F_n)[{\tors}] \cong E'_{n'}(F_{n'})[\tors] \cong G$ (say). Then both $n$ and $n'$ divide the exponent of $G$. The exponent of $G$ is bounded by $\#G$, and from \cite{CP15}, $\#G \le T_{\CM}(d) \ll x \log\log{x}$. (We could avoid appealing to \cite{CP15} by referencing earlier results of Silverberg \cite{silverberg92} or Prasad--Yogananda \cite{prasad}.) Hence (for large $x$) the number of integers that divide the exponent of $G$ is no more than $x^{\delta/2}$. So no more than $x^{\delta/2}$ numbers $n \in \N$  share the same value of $E'_{n}(F_n)[{\tors}]$. Thus,
\[ \#\T_{\rm CM}(d) \ge \frac{\#\N}{x^{\delta/2}} > x^{\delta/2}. \]
Since $d\le x$, this establishes Theorem \ref{thm:largeeven} with $\eta = \delta/2$.\end{proof}

We can go a bit further. The arguments presented by Pomerance in \cite[\S4]{pomerance89} suggest the following conjecture on the upper order of $\#\T_{\CM}(d)$.

\begin{conj}\label{conj:maxgroups} As $x\to\infty$,
\[ \max_{d \le x} \#\T_{\CM}(d) = x/L(x)^{1+o(1)}, \]
where $L(x) = \exp(\log{x} \frac{\log\log\log{x}}{\log\log{x}})$.
\end{conj}

Indeed, the proof of \cite[Theorem 4.4]{pomerance89} shows that under a reasonable conjecture on the distribution of smooth shifted primes (appearing there as Hypothesis 4.3), there are numbers $m \le x$ with at least $x/L(x)^{1+o(1)}$ representations in the form $\varphi(n)$, with $n$ squarefree. We expect Hypothesis 4.3 to remain true even when restricted to primes from the residue class $1\bmod{3}$. Combining the argument for \cite[Theorem 4.4]{pomerance89} with the above proof of Theorem \ref{thm:largeeven} then shows that
\[ \max_{d \le x} \#\T_{\CM}(d) \ge x/L(x)^{1+o(1)}. \]

The upper bound half of Conjecture \ref{conj:maxgroups} can be proved unconditionally. We start from \cite[Theorem 4.1]{pomerance89}, which asserts that 
\begin{equation}\label{eq:maxphipreimage}\max_{m \le x} \#\{n: \varphi(n)=m\} \le x/L(x)^{1+o(1)}. \end{equation}
Now take any positive integer $d \le x$. From \cite[Theorem 2.4]{BCP}, if $G \cong E(F)[\tors]$ for some CM elliptic curve $E$ over some degree $d$ number field $F$, then, writing $\rad(\cdot)$ for the product-of-distinct-prime-factors function,
\[ \varphi(\rad(\#G))= \prod_{\ell \mid \#G} (\ell-1) \mid 12d.\]
From the maximal order of the divisor function, $\tau(12d) \le \exp(O(\log{x}/\log\log{x}))$, and this last expresion is $L(x)^{o(1)}$. So given $d$, the integer $\varphi(\rad(\#G))$ is restricted to a set of at most $L(x)^{o(1)}$ possible values. It now follows from \eqref{eq:maxphipreimage} that, given $d$, there are at most $x/L(x)^{1+o(1)}$ possible values of $\rad(\#G)$. From \cite{CP15}, we have $\#G=O(x\log\log{x})$. As a consequence, given a value of $\rad(\#G)$, there are at most $\exp(O(\log{x}/\log\log{x}))$ possibilities for $\#G$ (see, e.g., \cite[Lemma 4.2]{pollack11}). Hence, there are no more than $x/L(x)^{1+o(1)}$ possibilities for $\#G$. But the structure of $G$ is determined by its two invariant factors; hence, given $\#G$, the number of possible choices for $G$ is crudely bounded by $\tau(\#G)$, which is at most $\exp(O(\log{x}/\log\log{x}))$. We conclude that $\#\T_{\CM}(d)\le x/L(x)^{1+o(1)}$, as claimed.

\section{CM Torsion Subgroups in Odd Degree $d \leq 99$}
We have written \texttt{PARI}/\texttt{GP} code which, given an odd positive integer $d$, returns the list of groups which appear the torsion subgroup of a CM elliptic curve defined over a number field of degree $d$. This code is available at the research website of either author. 

On a modern desktop computer, one can process all odd $d\le 2\cdot 10^8$ in about 12 hours. The output in the more modest range $d \leq 99$ is included below.

{\small
\begin{longtable}{c | p{.75\textwidth} }

Degree & Torsion Subgroups Appearing  \\ \midrule
\endhead
1  & $\Z/m\Z$ for $m=1,2,3,4,6$ and $\Z/2\Z \oplus \Z/2\Z$  \\ [.5 ex]
3  & $\Z/m\Z$ for $m=1,2,3,4,6,9,14$ and $\Z/2\Z \oplus \Z/2\Z$\\ [.5 ex]
   5    & $\Z/m\Z$ for $m=1,2,3,4,6,11$ and $\Z/2\Z \oplus \Z/2\Z$ \\ [.5 ex]
   7 & Olson\\ [.5 ex]
   9 & $\Z/m\Z$ for $m=1,2,3,4,6,9,14,18,19,27$ and $\Z/2\Z \oplus \Z/2\Z$ \\ [.5 ex]
    11 & Olson\\ [.5 ex]
   13 & Olson\\ [.5 ex]
   15 & $\Z/m\Z$ for $m=1,2,3,4,6,9, 11,14,22$ and $\Z/2\Z \oplus \Z/2\Z$ \\ [.5 ex]
   17 & Olson\\ [.5 ex]
   19 & Olson\\ [.5 ex]
   21 & $\Z/m\Z$ for $m=1,2,3,4,6,9,14,43$ and $\Z/2\Z \oplus \Z/2\Z$ \\ [.5 ex]
   23 & Olson\\ [.5 ex]
   25 & 5-Olson\\ [.5 ex]
   27 & $\Z/m\Z$ for $m=1,2,3,4,6,9,14,18,19,27,38,54$ and $\Z/2\Z \oplus \Z/2\Z$ \\ [.5 ex]
   29 & Olson \\ [.5 ex]
   31 & Olson \\ [.5 ex]
   33 & $\Z/m\Z$ for $m=1,2,3,4,6,9,14,46,67$ and $\Z/2\Z \oplus \Z/2\Z$ \\ [.5 ex]
   35 & 5-Olson\\ [.5 ex]
   37 & Olson \\ [.5 ex]
   39 & 3-Olson \\ [.5 ex]
   41 & Olson \\ [.5 ex]
   43 & Olson \\ [.5 ex]
   45 & $\Z/m\Z$ for $m=1,2,3,4,6,9,11,14,18,19,22,27,62$ and $\Z/2\Z \oplus \Z/2\Z$  \\ [.5 ex]
   47 & Olson \\ [.5 ex]
   49 & Olson \\ [.5 ex]
   51 & 3-Olson \\ [.5 ex]
   53 & Olson \\ [.5 ex]
   55 & 5-Olson \\ [.5 ex]
   57 & 3-Olson \\ [.5 ex]
   59 & Olson \\ [.5 ex]
   61 & Olson \\ [.5 ex]
   63 & $\Z/m\Z$ for $m=1,2,3,4,6,9,14,18,19,27,43,86$ and $\Z/2\Z \oplus \Z/2\Z$  \\ [.5 ex]
   65 & 5-Olson \\ [.5 ex]
   67 & Olson \\ [.5 ex]
   69 & 3-Olson \\ [.5 ex]
   71 &Olson\\ [.5 ex]
   73 & Olson \\ [.5 ex]
   75 & 15-Olson \\ [.5 ex]
   77 & Olson \\ [.5 ex]
   79 & Olson \\ [.5 ex]
   81 & $\Z/m\Z$ for $m=1,2,3,4,6,9,14,18,19,27,38,54,81,163$ and $\Z/2\Z \oplus \Z/2\Z$ \\ [.5 ex]
   83 & Olson  \\ [.5 ex]
   85 & 5-Olson  \\ [.5 ex]
   87 & $\Z/m\Z$ for $m=1,2,3,4,6,9,14,59$ and $\Z/2\Z \oplus \Z/2\Z$  \\ [.5 ex]
   89 & Olson  \\ [.5 ex]
   91 & Olson  \\ [.5 ex]
   93 & 3-Olson  \\ [.5 ex]
   95 & 5-Olson  \\ [.5 ex]
   97 & Olson  \\ [.5 ex]
   99 & $\Z/m\Z$ for $m=1,2,3,4,6,9,14,18,19,27,46,67,134$ and $\Z/2\Z \oplus \Z/2\Z$ \\ [.5 ex]
\end{longtable}}

\subsection*{Acknowledgements} We thank Kenneth Jacobs for suggesting we study large values of $\#\T_{\CM}(d)$ and Anton Mosunov for helpful pointers to the literature on unconditional class number computations. We thank Pete L. Clark for useful conversations and for suggesting we consider the stratification of torsion subgroups. We are also grateful to Mits Kobayashi for enlightening discussions about densities of sets of multiples. The research of the first author is supported in part by NSF grant DMS-1344994 (RTG in Algebra, Algebraic Geometry, and Number Theory at the University of Georgia). Work of the second author is supported by NSF award DMS-1402268.

\bibliographystyle{amsplain}
\bibliography{BP}

\end{document}